\documentclass[a4paper]{amsart}

\usepackage[T1]{fontenc}
\usepackage[latin1]{inputenc}
\usepackage[backend=biber,style=numeric, doi=false,isbn=false,url=false, maxnames=10, minnames = 9]{biblatex}
\usepackage[shortlabels,inline]{enumitem}
\usepackage{amsfonts}
\usepackage{amssymb}
\usepackage{amsmath}
\usepackage{amsthm}
\usepackage{faktor}
\usepackage[x11names,dvipsnames]{xcolor}
\usepackage{graphicx}%
\usepackage{tikz-cd}
\usepackage{tikz}\usetikzlibrary{nfold}
 
\usepackage[pdftex,linktocpage,breaklinks]{hyperref}
\hypersetup{
colorlinks=true,allcolors=blue,
linkcolor=blue,
citecolor=red,
anchorcolor=blue,
urlcolor=blue,
runcolor=red,
filecolor=red
pdfborder=0 0 0,
}
\PassOptionsToPackage{unicode}{hyperref}
\usepackage{algorithm}
\usepackage{algpseudocode}

\numberwithin{equation}{section} 

\theoremstyle{plain}
\newtheorem{FactCounter}{dummy}[section]
\newtheorem{Theorem}[FactCounter]{Theorem} %

\newtheorem{Proposition}[FactCounter]{Proposition} %
\newtheorem{Lemma}[FactCounter]{Lemma} %
\newtheorem{Corollary}[FactCounter]{Corollary} %
\theoremstyle{definition}
\newtheorem{Definition}[FactCounter]{Definition} %
\theoremstyle{remark}
\newtheorem{Remark}[FactCounter]{Remark} %
\newtheorem{Example}[FactCounter]{Example} %
\newcommand\operator[1]{\mathop{\operatorname{#1}}\nolimits}

\newcommand{\source}{\textbf{s}}
\newcommand{\range}{\textbf{r}}
\DeclareMathOperator{\lcm}{lcm}
\newcommand{\supp}{\operator{supp}}
\newcommand{\Stab}{\operator{Stab}}
\newcommand{\NN}{\mathbb{N}}
\newcommand{\CC}{\mathbb{C}}
\newcommand{\ZZ}{\mathbb{Z}}
\newcommand{\MatZ}{\text{M}_{N}(\ZZ)}
\newcommand{\MatN}{\text{M}_{N}(\NN)}

\newcommand{\TightIdeal}{{T}_{K}}
\newcommand{\SingularIdeal}{{I}_{K}}
\newcommand{\SingIdeal}{\SingularIdeal}

\newcommand{\SingularNotTightIdeal}{\SingularIdeal\setminus\TightIdeal}
\DeclareMathOperator{\Support}{\text{supp}}
\newcommand{\G}{\mathcal G}
\newcommand{\LongComment}[1]{\Comment{\parbox[t]{4cm}{#1}}}
\renewcommand{\int}{\operatorname{int}}
\newcommand{\onto}{\twoheadrightarrow}
\renewcommand{\hat}{\widehat}

\title{On the simplicity of Katsura algebras}

\author[Josiah Aakre]{Josiah Aakre}
\address{Department of Mathematics, University of Manchester, Manchester M13 9PL, United Kingdom.}
\email{josiah.aakre@postgrad.manchester.ac.uk}

\author[N\'ora Szak\'acs]{N\'ora Szak\'acs}
\address{Department of Mathematics, University of Manchester, Manchester M13 9PL, United Kingdom.}
\email{nora.szakacs@manchester.ac.uk}



\begin{document}

\maketitle
\begin{abstract}
	 We give a complete characterization of the (purely infinite) simplicity of Katsura algebras and the associated Steinberg algebras. This is achieved by characterizing when the singular ideals vanish via the self-similar groupoid model derived from Exel and Pardo. Analogous results are given for the algebras arising from the faithful quotient of the self-similar action. Finally, we describe polynomial-space algorithms to determine if each singular ideal vanishes and provide the first non-Hausdorff examples of non-contracting self-similar groupoids for which simplicity is algorithmically decidable.
\end{abstract}

\section{Introduction}

A key milestone in the classification theory of $C^{\ast}$-algebras is the Kirchberg-Phillips theorem, stating that all nuclear, separable, purely infinite simple $C^{\ast}$-algebras in the UCT class -- \emph{UCT Kirchberg algebras} for short -- are classified by their $K$-theory. 
In an influential paper \cite{Katsura}, Katsura defined a family of $C^{\ast}$-algebras $\mathcal O_{A,B}$ built from a pair of integer matrices $A$ and $B$ which contains all UCT Kirchberg algebras, moreover, the $K$-theory of $\mathcal O_{A,B}$ is easily computed from the matrices $A$ and $B$. The algebra $\mathcal O_{A,B}$ is unital exactly if $A, B$ are finite dimensional. While all Katsura algebras are nuclear, separable and in the UCT class, they are not always purely infinite simple, and the main goal of this paper is to characterize when this happens in terms of the matrices $A$ and $B$, focusing on the unital case.

Exel and Pardo realized \cite{ExParSSGraphs} that unital Katsura algebras can be viewed as $C^{\ast}$-algebras associated to self-similar actions of $\mathbb Z$ on the path spaces of finite graphs. 
This led them to introduce the notion of self-similar group actions on graphs,
and associated $C^{\ast}$-algebras, which generalize $C^{\ast}$-algebras of self-similar groups introduced by Nekrashevych \cite{Nek05SSGroups} in the same way that graph $C^{\ast}$-algebras generalize Cuntz algebras. This is further extended by \cite{LACA2018268} where the acting group is more generally replaced by a (discrete) groupoid. In this setting, Katsura algebras arise from a self-similar action of a bundle of infinite cyclic groups, termed Katsura-Exel-Pardo- or \emph{KEP-groupoids} for short, and this has now become the standard framework in which they are studied \cite{HUME_WHITTAKER_2025,ortega2020homology, NylandOrtegaKEP, miller2025homologyktheoryselfsimilaractions}.

Crucially, the approach of Exel and Pardo also gave Katsura algebras an \emph{ample groupoid} model (not to be confused with the self-similar groupoid mentioned above). The simplicity of nuclear Hausdorff groupoid $C^{\ast}$-algebras is well-understood, and was used in \cite[Section 18]{ExParSSGraphs} to study when $\mathcal O_{A,B}$ is purely infinite simple, obtaining a full classification in the Hausdorff case.

Katsura algebras do not always come from Hausdorff groupoids however, and simplicity in the non-Hausdorff setting is much more challenging to understand, the difficulty being that the $C^{\ast}$-algebra contains non-continuous functions. A general theory of non-Hausdorff simplicity developed in \cite{ClarkExelPardoSimsStarling2019} identifies the so-called \emph{singular ideal} as the main obstacle. 
Despite much subsequent research towards understanding when the singular ideal of a groupoid $C^{\ast}$-algebra vanishes, this remains a difficult property to pin down, and no characterization exists which is easy to verify on concrete examples. 

A significant breakthrough in \cite{brix2025hausdorffcoversnonhausdorffgroupoids} associates to each non-Hausdorff groupoid a `Hausdorff cover', providing a bridge from the non-Hausdorff setting to the Hausdorff case. Using this bridge, \cite{hume2025characterizationszerosingularideal, gonzales2026densesubalgebrassingularideal} show that in a large class of ample groupoids, the singular ideal is detected by a dense $*$-subalgebra called the (complex) Steinberg algebra of the groupoid, and in particular the simplicity of the Steinberg algebra and the $C^{\ast}$-algebra coincide. In explicit examples, the most successful approach to obtain a characterization of $C^{\ast}$-simplicity has been via the complex Steinberg algebra \cite{Gardella2025, AAKRE2026130446}. 

Steinberg algebras of ample groupoids can be defined over any field, and have been extensively studied in their own right, in particular they encompass all inverse semigroup algebras. Curiously, their simplicity can depend on the characteristic of the field \cite{ClarkExelPardoSimsStarling2019, SteSzaEtale}. Utilizing the results of \cite{gonzales2026densesubalgebrassingularideal}, the simplicity of Katsura algebras can be determined via the associated complex Steinberg algebra. Nevertheless, our methods extend seamlessly to Steinberg algebras over any field, so we frame our work in this broader setting.

Our main result about unital Katsura algebras is Theorem \ref{Thm:unfaithful J=0}, which gives a complete characterization of when the singular ideal of the Steinberg algebra vanishes. This is independent of the field and equivalent to the singular ideal vanishing in the Katsura algebra. We apply the result to obtain a complete characterization of simplicity of Katsura algebras in Corollary \ref{Cor: nonfaithful simplicity main result}. In Section \ref{Sec: complexity analysis-faithful}, we show that the conditions of Theorem \ref{Thm:unfaithful J=0} are effectively decidable, namely in polynomial space on the input $A, B$. By known results it follows that the simplicity of the Katsura algebra (and equivalently, any Steinberg algebra) is also decidable in polynomial space.

The self-similar action of a KEP-groupoid may fail to be faithful, in which case we refer to the corresponding quotient as the \emph{faithful KEP-groupoid}. The associated $C^{\ast}$-algebra is a natural quotient of the Katsura algebra, and was first studied in \cite{HUME_WHITTAKER_2025}. It turns out that the singular ideal of the Katsura algebra always arises from the non-faithfulness of the action, in particular, whenever the action is faithful and the associated ample groupoid is minimal and topologically free, the Katsura algebra is simple.

This naturally motivates the study of when \emph{faithful KEP-groupoid} ($C^{\ast}$)-algebras have a vanishing singular ideal. This turns out more difficult to characterize, but we obtain analogous results to the non-faithful case. Theorem \ref{Thm: testing singular elements faithful} provides a characterization of when the singular ideal of the Steinberg algebra vanishes, which again is independent of the field and equivalent to the vanishing of the singular ideal in the $C^*$-algebra. This is applied to obtain our simplicity result for faithful KEP-groupoid ($C^{\ast}$)-algebras in Corollary \ref{Cor: faithful simplicity main result}. The singular ideal characterization is again decidable in polynomial space on the input $A,B$ (Section \ref{Sec: complexity analysis-unfaithful}), and thus by known results, so is the simplicity. We remark that in the setting of finite matrices, simplicity implies pure infiniteness for both Katsura algebras and their faithful quotients.

In both cases, the singular ideal of the Steinberg algebra is studied by viewing it as the tight quotient of an inverse semigroup algebra, and using the characterization of the singular ideal in this language given in \cite{SteSzaEtale}. We also utilize ideas from \cite{Gardella2025} and \cite{AAKRE2026130446}, which study the simplicity of algebras associated to \emph{contracting} self-similar actions. KEP-groupoid actions are typically not contracting, and as such comprise the first family of non-contracting self-similar actions where simplicity is explicitly characterized. 

The paper is organized as follows. In Section 2, we review some preliminary results. Section 3 shows that the singular ideal of both the `usual' KEP-groupoids and their faithful quotients is detected by the complex Steinberg algebra. Section 4 proves a number of statements about KEP-groupoids which are useful in both the faithful and non-faithful settings, in particular we show that the faithful quotient of a KEP-groupoid can be computed in polynomial space on input $A,B$, which may be of independent interest.
Sections 5 and 6 outline the theoretical underpinnings required to decide the vanishing of the singular ideal in the `usual' and faithful cases respectively. Section 7 contains the formal algorithms and complexity analysis. Section 8 is dedicated to examples, in particular we give examples of KEP-groupoids as well as faithful KEP-groupoids with nontrivial singular ideals, which seems to be new in the literature. We also demonstrate that there is no implication between the singular ideal vanishing in a Katsura algebra and in its faithful quotient.

\section{Preliminaries}
\subsection{Graphs and self-similar groupoids}
Let $\mathbb{N}$ denote the nonzero natural numbers. Throughout, let $N\in \NN$, and let $[1,N]:=\{1,2, \ldots, N\}$. Let $\MatN$ and $\MatZ$ be the set of $N\times N$ matrices with entries in $\NN$ and $\ZZ$ respectively. 

In this paper, a directed graph $E=(E^0, E^1, \range, \source)$ consists of a finite set of vertices $E^0$, a finite set of edges $E^1$, and maps $\range,\source: E^1\rightarrow E^0$ assigning the range and source of each edge. Following the usual convention in dynamical systems, a path $p=e_1\ldots e_n$ in $E$ is a sequence of edges in $E^1$ such that $\source(e_{i})=\range(e_{i+1})$ for $1\leq i < n$.  We denote the length of a path $p$ by $|p|$, and the set of all paths in $E$ of length $n$ is denoted by $E^n$. We consider vertices to be paths of length 0, called empty paths. The set of all finite length paths is denoted by $E^*$ and the range and source maps extend to $E^*$ by defining $\range(e_1\ldots e_n) = \range(e_1)$ and $\source(e_1\ldots e_n) = \source(e_n)$, or in the case of $v\in E^0$, we set $\range(v)=v=\source(v)$. A cycle is a path $p$ satisfying $|p|>0$ and $\range(p)=\source(p)$. A path is simple if it does not contain a cycle as a subsequence of edges. A cycle is simple if it does not contain a cycle as a proper subsequence of edges. Path concatenation is a natural partial operation on $E^*$, and whenever we write $pq\in E^*$, it is implied that $p,q\in E^*$ and $\source(p)=\range(q)$. Similarly, for $p,q \in E^*$, we write
\[pE^*:= \{pw \colon w \in E^*, \source(p)=\range(w)\},\]
\[E^*q:= \{wq \colon w \in E^*, \source(w)=\range(q)\}.\]
Vertices are idempotent under path concatenation, and for any path $p\in E^*$, we have $p= \range(p)p\source(p)$.

We denote the set of infinite paths with range but no source by
\[E^{\omega}=\{e_{1}e_{2}e_{3}\ldots \colon e_i\in E^1, \source(e_{i})=\range(e_{i+1})\text{ for all }i\},\]
and analogously define $pE^{\omega}$ for $p\in E^*$.
Define $E^{-1}=\{e^{-1} \colon  e\in E^1\}$ to be a set in bijection with (but disjoint from) $E^1$.

In the paper, graphs will be given by their adjacency matrices (again we follow the convention of dynamical systems, and consider the transpose of the usual adjacency matrix in graph theory). Given a matrix $A\in \MatN$, let $E_A$ be a directed graph with vertices $[1,N]$ and for each $u,v\in[1,N]$, directed edges $e_{u,v,r}$ for each $0\leq r < A_{u,v}$. The incidence is given by $\range(e_{u,v,r})=u$ and $\source(e_{u,v,r})=v$. 

Recall that a groupoid is a small category in which every morphism has an inverse. The objects of a groupoid $G$ are denoted by $G^{0}$ and are sometimes referred to as units. For each object $v$, the set of isomorphisms of $v$ forms a group called the \emph{isotropy group} of $v$, denoted by $(G)_v$. Identifying each object of $G$ with the identity of its isotropy group, we view $G^{0}$ as a subset of $G$. A groupoid is called a \emph{group bundle} if it is the union of its isotropy groups.

\emph{Self-similar groupoids}, first defined in \cite{LACA2018268} as generalizations of self-similar groups, are groupoids acting on the path space of (usually finite) graphs. Given a directed graph $E$, a self-similar groupoid $(G, E)$ consists of a groupoid $G$ with object set identified with $E^0$, equipped with an action of $G$ on $E^*$ such that for all $e\in E^1$ and $g\in G\range(e)$, there exists a unique $g|_e\in G$ such that
\[g(ep) = g (e)g|_e(p),\]
for all $p\in \source(e)E^*$. Note that we do not assume the action to be faithful. We call the element $g|_e$ the section of $g$ at $e$. Sections may be extended to finite paths by recursively defining
\[g|_{ep}=(g|_e)|_p,\]
for all $ep\in E^*$. 

\smallskip

\textbf{Standing assumption:} throughout the paper, we will assume that the graph $E$ is finite as is suitable for our algorithmic methods. Because the graph is finite, all algebras considered are unital. Everywhere except Section \ref{Subsec: addressing sources} we also assume $E$ has no sources.

\smallskip

\subsection{The simplicity of inverse semigroup and ample groupoid algebras}
\label{Subsec: semigroup and groupoids algebras}

Recall that a semigroup $S$ is called inverse if for every element $s \in S$ there is a unique element $s^*$ such that $ss^*s=s$ and $s^*ss^*=s^*$.
Every self-similar groupoid gives rise to an inverse semigroup which we denote by $S_{(G,E)}$, which is generated by $E^0\cup E^1\cup E^{-1}\cup G$ together with a zero $0$. We set $g^*:=g^{-1}$, $e^*:=e^{-1}$ and $v^*:=v$ for $g\in G$, $e\in E^1$, and $v\in E^0$. The operation of $S_{(G, E)}$ extends that of $G$ by defining any undefined products in $G$ as 0. The following additional relations are imposed:
\begin{enumerate}
    \item $e\source(e) = \range(e)e = e$ for all $e\in E^0\cup E^1\cup E^{-1}$;
    \item $uv=0$ for distinct $u,v\in E^0$;
    \item $e^* f=0$ for distinct $e,f\in E^1$;
    \item $e^* e= \source(e)$ for $e\in E^1$;
    \item $ge = g(e)g|_e$ for all $g\in G$ and $e\in \source(g)E^1$.
\end{enumerate}
By a formal induction, one can extend (5) for paths: $gp = g(p)g|_p$ for all $g\in G$ and $p\in \source(g)E^*$.
It can easily be shown that non-zero elements of $S_{(G, E)}$ can be uniquely expressed as $pgq^*$ with $p,q\in E^*$, $\source(p)=\source(q)$, and $g\in G\source(p)$, which we refer to as the normal form. In particular, we view $E^* \cup \{0\}$ as a submonoid of $S_{(G,E)}$.

Extending the action of $G$, the inverse semigroup $S_{(G, E)}$ acts naturally on both $E^*$ and $E^{\omega}$ by partial bijections, 
given by
\begin{align*}
	pgq^* \colon qE^* \cup qE^\omega &\to pE^* \cup pE^\omega,\\
qw &\mapsto pg(w).
\end{align*}

Recall that a topological groupoid is a groupoid where all operations (including the source and range maps) are continuous. A groupoid is called \emph{étale} if the unit space is locally compact and Hausdorff, and the source and range maps are local homeomorphisms. An \emph{ample groupoid} is an étale groupoid with a totally disconnected unit space.

The inverse semigroup $S_{(G, E)}$ gives rise to the ample groupoid $\G_{(G, E)}$ as its tight groupoid (in the sense of Exel \cite{ExelCombinatorialalgebras}). Explicitly,
the groupoid $\G_{(G, E)}$ has unit space $E^\omega$ (equipped with the topology generated by cones $pE^\omega$ for $p  \in E^*$)\footnote{Here we use the fact that $E$ has no sources.}, and morphisms
\[\faktor{\{(s,w)\in S\times E^{\omega} \colon  0\neq s=pgq^*, w\in qE^{\omega}\}}{\sim},\]
where $(s,w)\sim(t,z)$ if and only if $w=z$ and there exists some prefix $v$ of $w$ such that $sv=tv$. The $\sim$-equivalence class of $(s,w)$ is denoted by $[s,w]$ and is also called the germ of $(s,w)$. The partial multiplication of $\G_{(G, E)}$ is given by $[s,t(w)][t,w]=[st,w]$, the inverse given by $[s,w]^{-1}=[s^*,s(w)]$, and the source map given by
$\source([s,w])=w$, which we can naturally identify with the identity morphism $[s^*s,w]$. The topology is the quotient topology inherited from $G \times E^\omega$, where $G$ is discrete.

In \cite{STEINBERG2010689}, Steinberg introduced the algebraic analogue of groupoid $C^{\ast}$-algebras, now termed \emph{Steinberg algebras}. We do not require the exact definition here; it suffices to know that the Steinberg algebra $K\G$ over the field $K$ consists of certain compactly supported (but not necessarily continuous) $\G \to K$ functions. The \emph{(reduced) $C^{\ast}$-algebra $C_r^*(\G)$ of $\G$} (in the sense of Paterson \cite{paterson1999groupoids}) can be obtained as the completion of $\mathbb C \G$ in reduced norm. (If $\G$ is amenable, as is the case for the groupoids we consider, then the reduced norm coincides with the universal norm.) 
Elements of $C_r^*(\G)$ can also be seen as $\G \to \mathbb C$ functions via Renault's $j$-map \cite{renault1980groupoid}. Given a function $f \colon \G \to K$, we denote its \emph{strict support} by $\supp(f)$.
We recall the following, important theorems characterizing the simplicity of $K \G$ and $C_r^*(\G)$.

\begin{Theorem}[{\cite[Thm. 4.16]{SteinbergandNora2023}, \cite[Thm. 7.26]{KwasniewskiMeyer2021}, see also \cite[Thm. 4.10]{ClarkExelPardoSimsStarling2019}}]\label{Thm: ample groupoid simplicity characterisations}
Let $\G$ be an amenable ample groupoid. Then
\begin{enumerate}
	\item The algebra $K\G$ is simple if and only if $\G$ is minimal, topologically free, and the ideal 
	$$J_K:=\{f \in K\G \colon \int(\supp(f))=\emptyset\}$$
	vanishes.
	\item The $C^{\ast}$-algebra $C^{\ast}_r(\G)$ is simple if and only if $\G$ is minimal, topologically free, and the ideal 
	$$J:=\{f \in C^{\ast}_r(\G) \colon \int(\supp(f))=\emptyset\}$$
	vanishes.
\end{enumerate}
\end{Theorem}
(We remark that only the `left-to-right' direction of item (2) uses amenability, and in fact $C_r^*(\G)=C^{\ast}(\G)$ is sufficient.)
The ideals $J_K$ and $J$ are called the ideals of singular functions, or \emph{singular ideals}. Given a groupoid, whether $J_K=0$ can depend on the characteristic of the field $K$, but it is an open question if $J_{\mathbb C}=0$ and $J=0$ are equivalent. 

For ample groupoids associated to self-similar actions, minimality and topological freeness are well understood. Our focus is on self-similar groupoids $(G,E)$ where $E$ is a finite graph with no sources and $G$ is a bundle of groups. In this setting, minimality of $\G_{(G,E)}$ is characterized by the following condition \cite[Prop. 6.32]{MundKwaTwisOpAlgSSG}:
\begin{itemize}
	\item [(Cof)] Each cycle of $E$ is in the same strongly connected component.
\end{itemize}
Topological freeness is equivalent to the conjunction of the following two conditions \cite[Prop. 6.24]{MundKwaTwisOpAlgSSG}:
\begin{itemize}
	\item [(Evr)] If $g$ fixes every path in $\source(g)E^*$ then there is $p\in\source(g)E^*$ such that $g|_p=\source(p)$.
	\item [(Cyc)] Each cycle of $E$ has an entrance.
\end{itemize}
When $C^{\ast}(\G_{(G,E)})$ is simple, \cite[Thm. 9.35]{MundKwaTwisOpAlgSSG} also characterizes its purely infinite simplicity by the following property:
\begin{itemize}
	\item [(Con)] Every vertex of $E$ is reachable from a cycle. 
\end{itemize}
Note that since all graphs we consider are finite and without sources, (Con) is always satisfied. 
In particular, we immediately obtain the following.
\begin{Proposition}
\label{Prop: simple implies purely infinite}
Let $E$ be a finite graph with no sources and $(G,E)$ a self-similar bundle of groups. Then if $C^{\ast}(\G_{(G,E)})$ is simple, it is also purely infinite simple.
\end{Proposition}
We remark that in \cite{MundKwaTwisOpAlgSSG}, (Cof), (Cyc) and (Con) take a more general form applicable to any self-similar groupoids.

Since the above conditions are well understood, the task of determining simplicity of $K\G_{(G,E)}$ or $C^{\ast}_r(\G_{(G,E)})$ reduces to deciding if $J_K$, respectively $J$, vanishes. In our setting, these each turn out to be equivalent to $J_\mathbb C=0$, further reducing the question to the algebraic setting (See Section \ref{Sec : algebraic and c* simplicity}).

Recall that we introduced $\G_{(G,E)}$ as the tight groupoid of the inverse semigroup $S_{(G,E)}$. By \cite[Cor. 2.14]{SteSzaEtale}, it follows that $K\G_{(G,E)}$ is isomorphic to the quotient of the contracted semigroup algebra $KS_{(G,E)}$ by the so-called \emph{tight ideal} $\TightIdeal$, which in this case coincides with the Cuntz-Krieger ideal generated by
$$v-\sum_{e \in\range^{-1}(v)} ee^* \ (v \in E^{0}).$$
This defines a morphism $$\varphi \colon KS_{(G,E)} \onto K\G_{(G,E)}$$
with kernel $\TightIdeal$, and the inverse image $\varphi^{-1}(J_K)$ is called the \emph{singular ideal} of the semigroup algebra, denoted $\SingIdeal$. It clearly follows that $J_K=0$ if and only if $\TightIdeal=\SingIdeal$. A key observation made in \cite[Cor. 6.4]{AAKRE2026130446} asserts that if $I_K \setminus T_K$ is non-empty, then it intersects the (discrete) groupoid algebra $KG$. 

For $$ a = \sum_{s \in S_{(G,E)}\setminus \{0\}} k_s s\in KS_{(G,E)},$$ we denote the set $\{s \in S_{(G,E)}\setminus \{0\} \colon k_s\neq 0\}$ by $\supp a$, and refer to it as the \emph{support}. Likewise, we say that a subset $X\subseteq S_{(G, E)}$  \emph{supports} $a$ or $a$ is supported on $X$ whenever $\supp a\subseteq X$.

Suppose $a \in KG \cap I_K \setminus T_K$. Observe that $a=\sum_{u\in E^0}\sum_{v\in E^0} uav$, where $uav$ is singular for all $u,v\in E^0$ (since $I_K$ is an ideal), and not tight for at least some $u,v\in E^0$ (since $T_K$ is closed under addition). Multiplying by any $g\in vGu$, we may find $guav$ is supported on the isotropy group $(G)_{v}$ and contained in $I_K \setminus T_K$.

In summary:

\begin{Proposition}
\label{prop:how_we_check_J=0}
Let $(G, E)$ be a self-similar groupoid. Then the singular ideal $J_K$ of $K\G_{(G,E)}$ is non-zero if and only if there exists some $v \in E^0$ such that $K(G)_v \cap I_K \setminus T_K \neq \emptyset,$ that is, $(G)_v$ supports elements in $KS_{(G,E)}$ which are singular but not tight.
\end{Proposition}

The quotient $KS_{(G,E)}/T_K$ is called the \emph{tight quotient} of $KS_{(G,E)}$, and elements of $T_K$ and $I_K$ are called \emph{tight} and \emph{singular} respectively.

We rely on the following descriptions of the tight and singular ideals of $KS_{(G,E)}$.

\begin{Proposition}[{\cite[Prop. 4.3]{AAKRE2026130446}}]
\label{prop:tight_characterization}
Let $a \in KS_{(G,E)}$. Then $a \in \TightIdeal$ if and only if every infinite path $p \in E^\omega$ has a prefix $q \in E^*$ such that $aq=0$ in $KS_{(G,E)}$.
\end{Proposition}

\begin{Proposition}[{\cite[Prop. 4.5]{AAKRE2026130446}}, {\cite[Prop. 4.9]{SteSzaEtale}}]
\label{prop:singular_characterization}
Let $a \in KS_{(G,E)}$. Then $a \in \SingIdeal$ if and only if for every finite path $p \in E^*$ there exists some $q \in s(p)E^*$ such that $apq=0$ in $KS_{(G,E)}$.
\end{Proposition}

We remark that \cite{AAKRE2026130446} defines self-similar groupoids as faithful, but this assumption is not used in the proofs of the above.

We now observe a simple necessary condition for the presence of singular elements which will be very useful later.

\begin{Proposition}\label{Prop: free orbits prevent singular elements}
	Let $H$ be a subgroup of a self-similar groupoid $(G,E)$ acting on $vE^*$, where $v \in E^0$. If the action of $H$ on $vE^\omega$ has a free orbit $O$, then $H$ does not support any non-zero singular elements in $KS_{(G, E)}$ for any field $K$.
\end{Proposition}
\begin{proof}
	Let $a:=\sum_{g\in\Support a}k_gg\in KH$. Because the action of $H$ on $O$ is free, and the support of $a$ is finite, any infinite path in $O$ has some prefix $p$ such that $g(p)\neq h(p)$ for any distinct $g,h\in \Support a$. Suppose that $a$ is singular. Then there exists some $q\in \source(p)E^*$ such that $apq=0$. It follows that
	\[0 = \sum_{g\in \Support a}k_gg(p)g|_p(q),\]
	where all the summands are in normal form and pairwise different. This implies that $\supp a=\emptyset$ and thus $a=0$.
\end{proof}

\subsection{Katsura algebras and Katsura-Exel-Pardo groupoids}

Katsura algebras were first defined in \cite{Katsura} as models of Kirchberg algebras. Exel and Pardo realized these as algebras of groups acting on graphs in a self-similar way \cite{ExParSSGraphs}, which can also be seen as self-similar groupoids. These have been termed as Katsura-Exel-Pardo self-similar groupoids (KEP-groupoids) in the literature, and are defined as follows.

Let $A\in \MatN$ have no zero rows, and let $B \in \MatZ$ such that $A_{u,v} = 0$ implies $B_{u,v}=0$.
We will call such $(A,B)$ a \emph{Katsura pair}.
Consider the groupoid with unit space $E_A^{0}$ consisting of a copy of the infinite cyclic group $\ZZ^{(v)}=\langle g_v \rangle$ for each $v\in E_A^0$ setting $\source(g_v)=\range(g_v)=v$ which is identified with $g^0_v$.
We define a self-similar action of the groupoid $\bigcup_{v \in E_A^{0}} \ZZ^{(v)}$ on $E_A^*$ as follows. For each $u,v\in [1,N]$, and each $k\in \ZZ$ and $r\in [0,A_{u,v}-1]$, there is a unique $\hat{k}\in \ZZ$ and $\hat{r}\in [0,A_{u,v}-1]$ such that
\begin{equation}\label{EQ: action equation}
    kB_{u,v}+r = \hat{k}A_{u,v}+\hat{r}.
\end{equation}
This governs the action of the groupoid on $E_A^*$ as follows:
for $g_u^k\in \ZZ^{(u)}$ and $e_{u,v,r}\in E_A^1$, define the action
\[g_u^k (e_{u,v,r})=e_{u,v,\hat{r}},\]
and the section
\[g_u^k|_{e_{u,v,r}}=g_v^{\hat{k}}.\]
This extends to a self-similar action on $E_A^*$ by letting $g(ep) := g(e)g|_e(p)$ for all $g\in \bigcup_{v \in E_A^{0}} \ZZ^{(v)}$ and  $ep\in \source(g)E_A^*$. We denote the associated self-similar groupoid by $(\hat G_B, E_A)$, and call it the \emph{KEP-groupoid associated to $A$ and $B$}. The associated $C^{\ast}$-algebra $C^{\ast}_r(\G_{(\hat G_B, E_A)})$ coincides with Katsura's algebra $\mathcal{O}_{A, B}$ \cite{ExParSSGraphs}. It is always nuclear, separable, satisfies the UCT \cite[Prop. 2.9]{Katsura}, and its $K$-theory is easily computable from the matrices $A$ and $B$ \cite[Prop. 2.6]{Katsura}. Our interest is in simplicity and vanishing of the singular ideal. Towards that end, sufficient conditions have been obtained in \cite[Prop. 2.10]{Katsura} and \cite[Thm. 18.12]{ExParSSGraphs}.

The action of $(\hat G_B, E_A)$ on $E_A^*$ is not always faithful.
Following \cite{HUME_WHITTAKER_2025}, by quotienting $(\hat G_B, E_A)$ with the kernel of the action, we also consider the associated faithful self-similar groupoid, denoted by $(G_B, E_A)$, and refer to it as the \emph{faithful KEP-groupoid} associated to $A$ and $B$. By slight abuse of notation, we denote the generators of both $(\hat G_B)_v$ and $(G_B)_v$ by $g_v$, their actions are governed by the same rules above, but the order of $(G_B)_v$ may be finite. 

\subsection{A remark on the assumption of no sources}\label{Subsec: addressing sources}

In this paper, we only consider Katsura pairs $(A, B)$ where $E_A$ has no sources to remain consistent with much of the existing literature \cite{Katsura, ExParSSGraphs, HUME_WHITTAKER_2025, LACA2018268, AAKRE2026130446}. The assumption of no sources has been dropped at times, see for example \cite{MundKwaTwisOpAlgSSG, miller2025homologyktheoryselfsimilaractions}. In the setting of Katsura algebras, it turns out that we can extend our classification of vanishing of the singular ideal to the setting with sources by adding loops to the sources.

Let $A^{\circ}\in \MatN$ have zero rows, and let $B^{\circ} \in \MatZ$ such that $A^{\circ}_{u,v} = 0$ implies $B^{\circ}_{u,v}=0$. Let $A$ and $B$ be the matrices obtained by copying $A^{\circ}$ and $B^{\circ}$ respectively, and then setting $A_{v,v}=1=B_{v,v}$ whenever $v$ is a source in $E_{A^{\circ}}$. The result is that each source in $E_{A^{\circ}}$ is given a loop in $E_{A}$ with associated $B$-value defined as 1. 

 Consider the faithful self-similar groupoid $(G_{B^{\circ}}, E_{A^{\circ}})$ defined by applying Equation \ref{EQ: action equation} to $(A^{\circ}, B^{\circ})$, and consider the faithful KEP-groupoid $(G_B, E_A)$. It is easy to show that the gauge invariant subgroupoid (see Section \ref{Sec : algebraic and c* simplicity}) associated to $( G_{B^{\circ}}, E_{A^{\circ}})$ coincides with that of $(G_{B}, E_{A})$. By Lemma \ref{Lem:J0 and J}, the  $C^{\ast}$-singular ideals of $\G_{(G_{B}, E_{A})}$ and $\G_{( G_{B^{\circ}}, E_{A^{\circ}})}$ vanish at the same time. This coincides with the vanishing of the algebraic singular ideals by Theorem \ref{Thm: J and JC}. These same observations apply to $(\hat G_{B^{\circ}}, E_{A^{\circ}})$ and $(\hat G_{B}, E_{A})$
 as well.
 
Note that the construction of adding loops does not always preserve minimality and topological freeness of the associated ample groupoids.

\section{Algebraic simplicity and $C^{\ast}$-simplicity coincide}
\label{Sec : algebraic and c* simplicity}

Before we describe the algorithms to determine simplicity, let us first show that simplicity of the complex Steinberg algebras of a (faithful or non-faithful) KEP-groupoid coincides with simplicity of the reduced $C^{\ast}$-algebra.

Let $(G,E)$ be a countable, self-similar bundle of abelian groups and $\G_{(G,E)}$ the associated ample groupoid.
Its \emph{gauge invariant subgroupoid} $\G_{0(G, E)}\subseteq\G_{(G, E)}$ is an open subgroupoid defined as
$$\G_{0(G, E)}=\{[rgq^*, p] \in \G_{(G, E)} \colon |r|=|q| \}.$$
By definition, $\CC(\G_{0(G,E)})$ is contained in $\CC(\G_{(G,E)})$ (tacitly extending functions on $\G_{0(G,E)}$ to $\G_{(G,E)}$ by $0$), and thus we also have $C_r^*(\G_{0(G,E)}) \subseteq  C_r^*(\G_{(G,E)})$ \cite[Remark 3.3]{Gardella2025}.

Let $J_0$ denote the singular ideal of $C_r^*(\G_{0(G,E)})$ and let $J_{0\mathbb{C}}$ denote the singular ideal of $\mathbb{C}\G_{0(G,E)}$. By definition, $J_{0\CC}=J_\CC \cap \mathbb{C}\G_{0(G,E)}$.

\begin{Lemma}\label{Lem: isotropies in gauge invariant subgroupoid}
The groupoid $\G_{0(G, E)}$ is amenable with abelian isotropies.	
\end{Lemma}

\begin{proof}
Observe that the discrete groupoid $G$ is amenable, since it is the union of abelian (in particular amenable) groups. It follows by \cite[Thm. 2.18]{miller2025homologyktheoryselfsimilaractions}	that $\G_{(G,E)}$ is amenable.
The amenability of $\G_{0(G, E)}$ follows from the amenability of $\G_{(G,E)}$ by  \cite[Prop. 2.17 (i)]{miller2025homologyktheoryselfsimilaractions}.

Now fix $p \in E^\omega$ and consider the isotropy group around $p$, denoted by $p\G_{0(G, E)}p$. Note that if $[rgq^*,p] \in p\G_{0(G, E)}p$, i.e. if $\range([rgq^*,p])=rg(q^*p)	=p$, then both $q$ and $r$ are prefixes of $p$ and of the same length. It follows that $r=q$. 
Now let $[qgq^*, p],[qrhr^*q^*, p]$ be two arbitrary elements of $p\G_{0(G, E)}p$. Then
\begin{align*}
	[qgq^*, p][qrhr^*q^*, p] &\overset{(1)}{=}[qrg|_rr^*q^*, p][qrhr^*q^*, p]\\
	&\overset{(2)}{=} [qrg|_rhr^*q^*, p]\\
	&\overset{(3)}{=} [qrhg|_rr^*q^*, p]\\
	&\overset{(4)}{=} 	[qrhr^*q^*, p][qgq^*, p],
\end{align*}
where (1) follows by $qgq^* \cdot qr=qgr=qg(r)g|_r=qrg|_r=qrg|_rr^*q^* \cdot qr$, (2) and (4) follow by definition and (3) from the assumption that $(G)_{\source(r)}$ is commutative. 
\end{proof}


\begin{Lemma}
	\label{Lem:J0 and J}
Let $(G,E)$ be a self-similar bundle of abelian groups. Then $J_0=0$ if and only if $J=0$.
\end{Lemma}

\begin{proof}
		By \cite[Lemma 3.5]{Gardella2025}, there is a faithful conditional expectation $E \colon C_r^*(\G_{(G,E)}) \to  C_r^*(\G_{0(G,E)})$ such that $E(J)=J_0$. Faithfulness here means that the kernel of $E$ contains no nontrivial ideal, in particular $J_0 \neq 0$ implies $J \neq 0$. The converse is implied by $J_0 \subseteq J$ which holds by definition.
\end{proof}

\begin{Theorem}
	\label{Thm: J and JC}
	Let $(G,E)$ be a self-similar bundle of abelian groups. Then the singular ideal $J$ vanishes if and only if the algebraic singular ideal $J_{\mathbb{C}}$ vanishes.
\end{Theorem}
\begin{proof}
	By definition, if $J=0$ then $J_{\CC}=0$. We only need to show the converse. 
	Suppose that $J_{\CC}=0$, thus in particular
	$J_{0\CC}=0$. By Lemma \ref{Lem: isotropies in gauge invariant subgroupoid}, we can apply \cite[Thm. D and Thm. G]{gonzales2026densesubalgebrassingularideal} to deduce that $J_0=0$. By Lemma \ref{Lem:J0 and J}, $J=0$.	
\end{proof}

\begin{Corollary}\label{Cor: cstar simplicity iff steinberg simplicity}
	Let $(A,B)$ be a Katsura pair.
	Then 
	\begin{enumerate}
		\item  $C_r^*(\G_{(\hat G_B, E_A)})$ is simple if and only if $\mathbb{C}\G_{(\hat G_B, E_A)}$ is simple.
		\item  $C_r^*(\G_{(G_B, E_A)})$ is simple if and only if $\mathbb{C}\G_{(G_B, E_A)}$ is simple.
	\end{enumerate}
\end{Corollary}


\section{Fixed paths}
\label{Sec: fixed paths}

In order to determine when singular ideals vanish in KEP groupoid algebras and their faithful counterparts, the first step is to understand when elements of KEP-groupoids fix paths.
 

Let $(A, B)$ be a Katsura pair and let $E_A$ be the corresponding graph. Following \cite[Section 18]{ExParSSGraphs}, for a non-empty path $p=e_1\ldots e_n\in E_A^*$, define
\[A_{p}=\prod_{1\leq l \leq n} A_{e_l} \text{ and }B_{p}=\prod_{1\leq l \leq n} B_{e_l},\]
where $A_{e_l}=A_{\range(e_l),\source(e_l)}$ and $B_{e_l}=B_{\range(e_l),\source(e_l)}$. For any $v \in E_A^0$ we put $A_v=B_v=1$. Observe that $A_p \neq 0$ for any path $p$ by definition. We will sometimes refer to $p$ as a \emph{zero path} if $B_p=0$ and a \emph{non-zero path} otherwise.

The following is easy to see by induction on the length of paths:
\begin{Lemma}[{\cite[Lemma 18.4]{ExParSSGraphs}}]\label{Lem:fixed path}
Let $(A,B)$ be a Katsura pair.
A path $p \in vE_A^*$ is fixed by $g_v^k$ if and only if $\frac{kB_q}{A_q} \in \mathbb Z$ for every prefix $q$ of $p$.
\end{Lemma}
 
It immediately follows:
\begin{Lemma}[{\cite[Prop. 3.2]{HUME_WHITTAKER_2025}}]\label{Lem: OG kernel description}
    Let $(A,B)$ be a Katsura pair. Then $g_v^k \in (\hat G_B)_v$ acts trivially on $E_A^*$ if and only if $\frac{kB_p}{A_p}\in \ZZ$ for all $p\in vE_A^*$.
\end{Lemma}

We apply the above proposition to find a \emph{computable} characterization of the sizes of the isotropy groups of faithful KEP-groupoids. We first need the following lemma.
\begin{Lemma}\label{Lem: factoring cycles into simple cycles}
	For any cycle $c\in E_A^*$, there exists simple cycles $c_1,\ldots, c_n\in E_A^*$ such that the concatenation of edge sequences $c_1c_2\ldots c_n$ is a permutation of the edge sequence of $c$. Consequently, $A_c = \prod_{i=1}^{n} A_{c_i}$ and $B_c = \prod_{i=1}^{n} B_{c_i}$.
\end{Lemma}
\begin{proof}
	Let $c=e_1 \ldots e_m\in E_A^*$ and suppose the statement holds for all cycles shorter than $c$. If $c$ is simple, the statement holds trivially. If $c$ is not simple, then $c$ is a cyclic permutation of the concatenation $d_1d_2$ of two cycles $d_1, d_2\in E_A^*$, in particular $A_c=A_{d_1}A_{d_2}$. The statement holds for $d_1$ and $d_2$, so there exists simple cycles $c_1,\ldots, c_i$ satisfying the statement for $d_1$ and simple cycles $c_{i+1},\ldots,c_n$ satisfying the statement for $d_2$. It follows that the simple cycles $c_1,\ldots,c_n$ establish the claim for $c$.
\end{proof}

\begin{Proposition}\label{Prop: Infinite group characterization}
  Let $(A,B)$ be a Katsura pair and $(G_B, E_A)$ the faithful KEP-groupoid. Let $v\in E_A^0$.
  \begin{enumerate}
  \item Then $(G_B)_v$ is infinite if and only if there is a simple path $p\in vE_A^*$ and a simple cycle $c\in \source(p)E_A^*$ such that $B_p\neq 0$ and $\frac{B_c}{A_c}\notin\ZZ$. In this case, the action of $(\hat G_B)_v=(G_B)_v$ has a free orbit.
  
  \item If $(G_B)_v$ is a finite group, the order is \[|(G_B)_v|= \lcm\left\{\frac{A_p}{\gcd(A_p, |B_p|)}\colon p\in vE_A^*\text{ is a simple path}, B_p\neq 0\right\}.\]
\end{enumerate}    
\end{Proposition}
\begin{proof}
   Let $p$ and $c$ be as in part (1). We will show that $(\hat G_B)_v$ acts freely on the orbit of $pc^\omega$, in particular implying that the action is faithful and $(\hat G_B)_v=(G_B)_v$. 
   By Lemma \ref{Lem:fixed path}, $g_v^k \in (\hat G_B)_v$ stabilizes $pc^\omega$ if and only if $\frac{kB_q}{A_q}\in \ZZ$ for all prefixes $q$ of $pc^\omega$. Consider the ratio 
   \[\frac{kB_{pc^n}}{A_{pc^n}}=\frac{kB_{p}}{A_{p}}\cdot\frac{B_{c^n}}{A_{c^n}} = \frac{kB_{p}}{A_{p}}\cdot\left(\frac{B_{c}}{A_{c}}\right)^n,\]
   for some natural number $n$. If $\frac{B_{c}}{A_{c}}\notin\ZZ$, there exists some natural number $n_0$ such that $\frac{kB_{p}}{A_{p}}\cdot\left(\frac{B_{c}}{A_{c}}\right)^n\notin\ZZ$ for all $n\geq n_0$. Thus, $g_v^k$ acts nontrivially on $pc^\omega$ for any $k \neq 0$ as needed.

   We prove the other direction by contrapositive. Suppose that for all $pc\in vE_A^*$ such that  $p$ is a simple path, $B_p\neq 0$ and $c$ is a simple cycle, we have $\frac{B_c}{A_c}\in \ZZ$. Notice that for any $qe\in vE_A^*$ with $B_e=0$, $g_v^k$ fixes $qe$ if and only if it fixes $q$, so it suffices to analyze the action on non-zero paths. We can factor any non-zero $q\in vE_A^*$ as $q=p_1c_1p_2\cdots p_nc_np_{n+1}$ where $c_i$ is a cycle and $p_1 \ldots p_np_{n+1}=:p$ is a simple path. By Lemma \ref{Lem: factoring cycles into simple cycles}, for each $l$, we have $\frac{B_{c_l}}{A_{c_l}} = \prod_{j=1}^{m}\frac{B_{d_j}}{A_{d_j}}$ where $d_1,\ldots,d_m$ are simple cycles with edges contained in $c_l$. Observe that $v$ is reachable from each $d_i$ along some non-zero path, and therefore $\frac{B_{d_j}}{A_{d_j}}\in \ZZ$ for all $j$.  We conclude that $\frac{B_{c_i}}{A_{c_i}}\in \ZZ$ for all $1\leq i\leq n$. Then 
    $$\frac{B_{q}}{A_q}=\frac{B_{p}}{A_{p}}\cdot\frac{B_{c_1}}{A_{c_1}} \dots \frac{B_{c_n}}{A_{c_n}}.$$
    In particular, $g_v^k$ fixes $q$ if and only if $k\frac{B_{p}}{A_{p}} \in \mathbb Z$.
    It follows that $g_v^k$ acts trivially if and only if $k\frac{B_{p}}{A_{p}} \in \mathbb Z$ for any simple path $p$ with $B_p \neq 0$, equivalently, if and only if $k$ divides $\lcm\left\{\frac{A_p}{\gcd(A_p, |B_p|)}\colon p\in vE_A^*\text{ is a simple path}, B_p\neq 0\right\}$. Since there are finitely many simple paths in a finite graph, this least common multiple is finite, implying the finiteness of $(G_B)_v$.
    
Moving on to part (2), it follows from the previous paragraph that
$$|(G_B)_v| = \lcm\left\{\frac{A_p}{\gcd(A_p, |B_p|)}\colon p\in vE_A^*\text{ is a simple path}, B_p\neq 0\right\}.$$
\end{proof}

It follows that on input $(A,B)$ and $1 \leq v \leq |E_A^{0}|$, one can effectively compute the order $|(G_B)_v|$ by first testing (1) to find out if $|(G_B)_v|$ is finite, and if so, using the formula in (2) to compute its order (see Algorithm \ref{alg:order of isotropies}).

We close the section with a technical lemma we will often use.
\begin{Lemma}
\label{lem: fixed path computation}
Let $(A,B)$ be a Katsura pair. Let $u,v\in E_A^0$, and suppose $p \in uE_A^*v$ and $g_u^k(p)=p$. Then
\begin{enumerate}
	\item in $(\hat G_B, E_A)$, we have $g_u^k|_p=g_v^l$ iff 
	$l = k \cdot \frac{B_p}{A_p};$
	\item  in $(G_B, E_A)$, we have $g_u^k|_p=g_v^l$ iff 
	$l \equiv k \cdot \frac{B_p}{A_p} \pmod m,$ where $|(G_B)_v|=m$.
\end{enumerate}
\end{Lemma}

\begin{proof}
Observe that (1) implies (2) by definition, thus we focus on proving (1), by induction on $|p|$. For $|p|=0$ the claim is obvious. Let $|p| \geq 1$, suppose $g_u^k(p)=p$ and let $g_v^l:=g_u^k|_p$.  Put $p=eq$ where $e \in E_A^1, q\in E_A^*$. Since $g_v^k(e)=e$, by (\ref{EQ: action equation}) we must have 
$kB_e=\hat k A_e$ where $g_u^k|_e=g_{s(e)}^{\hat k}$. Hence $\hat k=k\cdot \frac{B_e}{A_e}$. By induction, we also have $l=\hat k\cdot \frac{B_q}{A_q}$ which implies the claim.
\end{proof}

\section{The singular ideal associated to KEP-groupoids}
\label{Sec: nonfaithful}

In this section, we characterize when Katsura algebras have a vanishing singular ideal by using Propositions \ref{prop:tight_characterization} and \ref{prop:singular_characterization} to look for singular elements in the inverse semgiroup algebra which are not tight.

Throughout the section, let $(A,B)$ be a Katsura pair and $(\hat G_B, E_A)$ the associated KEP-groupoid, and let $v \in E_A^{0}$. An immediate consequence of Propositions \ref{Prop: free orbits prevent singular elements} and \ref{Prop: Infinite group characterization} is the following.

\begin{Proposition}
	\label{Prop: faithful groups nonsingular}
	Let $(\hat G_B, E_A)$ be a KEP-groupoid and $v \in E_A^{0}$. Suppose $(\hat G_B)_v$ acts faithfully on $E_A^*$. Then $(\hat G_B)_v$ does not support a non-zero singular element in $KS_{(\hat G_B, E_A)}$.
\end{Proposition}

We will see another easy-to-obtain (as well as to verify) necessary condition for a non-zero singular element. The key lemma is as follows.

\begin{Lemma}\label{Lem: ap=0 then Bp=0}
	Suppose $a\in K(\hat G_B)_v\setminus\{0\}$ and $p\in vE_A^*$. Then $ap=0$ implies $B_p=0$.
\end{Lemma}
\begin{proof}
	If $ap=0$ then we can find distinct $g_v^k, g_v^l\in \supp a$ such that $g_v^k(p)=g_v^l(p)$ and $g_v^k|_p=g_v^l|_p$, or equivalently, $g_v^{k-l}(p)=p$ and $g_v^{k-l}|_p=v=g_v^0$. By Lemma \ref{lem: fixed path computation}, since $k \neq l$, we must have $B_p=0$.
\end{proof}

In combination with Proposition \ref{prop:singular_characterization}, we obtain the following.

\begin{Proposition}
	\label{Prop: singular necessary unfaithful}
If $(\hat G_B)_v$ supports a non-zero singular element, then for all $p\in vE^*$ there exists some $q\in \source(p)E_A^*$ such that $B_{pq}=0$.
\end{Proposition}

Now let $\pi \colon (\hat G_B, E_A) \to (G_B, E_A)$ be the natural quotient map, and let $\hat T_K$ denote the tight ideal of $KS_{(\hat G_B, E_A)}$.
Let $\kappa$ be the kernel of $\pi$, i.e. the set $\{g \in \hat G_B \colon \pi(g)=\source(g)\}$, and define
$$\kappa_T:= \{g \in \hat G_B \colon g-\source(g) \in \hat T_K\}.$$
Following \cite{miller2025homologyktheoryselfsimilaractions} we call $\kappa_T$ the \emph{tight kernel} of the self-similar action. Of course, both $\kappa$ and $\kappa_T$ contain the units $E_A^0$ by definition. 

\begin{Lemma}
\label{Lem: tight kernel characterization}	
For $g\in \hat G_B \setminus \hat G_B^0$, we have $g \in \kappa_T$ if and only if $g \in \kappa$ and 
every infinite path $p \in \source(g)E_A^\omega$ has a prefix $q$ with $B_q=0$.
\end{Lemma}

\begin{proof}
By Proposition	\ref{prop:tight_characterization}, we have $g-\source(g) \in \hat T_K$ if and only if 
every infinite path $p \in E_A^\omega$ has a prefix $q$ with $gq=\source(g)q$.
If $p \notin \source(g)E_A^\omega$, then $gq=\source(g)q=0$ for any prefix of $p$, and if  $p \in \source(g)E_A^\omega$, then for any prefix $q$ of $p$ we have $gq=g(q)g|_q$ and $\source(g)q=q$. As both $g(q)g|_q$ and $q$ are in normal form, they are equal exactly if $g(q)=q$ and $g|_q=\source(q)$. By Lemma \ref{lem: fixed path computation}, if $g \neq \source(g)$ this is further equivalent to $g(q)=q$ and $B_q=0$. The claim follows.
\end{proof}

By the above, $\kappa_T \subseteq \kappa$, and so for each $v\in E_A^0$, exactly one of the following hold:
\begin{enumerate}
	\item $(\hat G_B)_v$ acts faithfully,
	\item $(\hat G_B)_v$ intersects $\kappa\setminus \kappa_T$,
	\item $(\hat G_B)_v$ intersects $\kappa_T$ nontrivially.
\end{enumerate}

We have seen that in case (1), $(\hat G_B)_v$ supports no nontrivial singular element. We shall now see what happens in the remaining cases.

\begin{Proposition}
	\label{Prop: nonfaithful2}
	Suppose $(\hat G_B)_v$ intersects $\kappa \setminus \kappa_T$. Then $K(\hat G_B)_v$ contains singular elements which are not tight if and only if for all $p\in vE_A^*$ there exists some $q\in \source(p)E_A^*$ such that $B_{pq}=0$.
\end{Proposition}
\begin{proof}
	The `left-to-right' direction follows immediately by Proposition \ref{Prop: singular necessary unfaithful}. For the converse, let $g \in (\hat G_B)_v \cap \kappa \setminus \kappa_T$, and consider $a=g-\source(g)$. By the definition of $\kappa_T$, we have $a \notin \hat T_K$. We will show that $a$ is singular using Proposition \ref{prop:singular_characterization}. Indeed, let $p \in E_A^*$ -- if $\range(p) \neq v$, then $ap=0$. Otherwise by assumption there exists some $q\in \source(p)E_A^*$ such that $B_{pq}=0$, in particular $g(pq)=pq$ and by Lemma \ref{lem: fixed path computation}, $g|_{pq}=\range(pq)$, so $apq=pq-pq=0$. It follows that $a$ is singular.	
\end{proof}

\begin{Proposition}
	\label{Prop: nonfaithful_3}
	Suppose $(\hat G_B)_v$ intersects $\kappa_T$ nontrivially. Then every singular element of $K(\hat G_B)_v$ is tight.
\end{Proposition}
\begin{proof}
	Suppose that $a\in K(\hat G_B)_v$ is singular, and let $p \in vE_A^\omega$ be an infinite path. By 
	Lemma \ref{Lem: tight kernel characterization}, there exists a finite prefix $q$ of $p$ with $B_q=0$.
	By the singularity of $a$, there is also some $r \in \source(q)E_A^*$ with $aqr=0$. Let $a=\sum_{g\in \supp a}k_g g$. Then
	\[0=aqr = \sum_{g\in \supp a}k_ggqr = \sum_{g\in \supp a}k_g g(q)r,\]
	which implies that $aq=\sum_{g \in \supp a}k_gg(q)=0$. Thus, $a$ is tight by Proposition \ref{prop:tight_characterization}.
\end{proof}

We obtain the following characterization of vanishing singular ideal.

\begin{Theorem}
	\label{Thm:unfaithful J=0}
Let $(A,B)$ be a Katsura pair and $(\hat G_B, E_A)$ the associated KEP-groupoid. Let $K$ be any field.
The singular ideal $\hat J_K$ of $K\G_{(\hat G_B, E_A)}$ is non-zero if and only if there exists some vertex $v \in E_A^{0}$ with $(\hat G_B)_v$ acting non-faithfully and such that
\begin{enumerate}
	\item[(T1)]  there exist a simple path $r \in vE_A^*$ and a simple cycle $c \in \source(r) E_A^* \source(r)$ with $B_{rc} \neq 0$; and
	\item[(T2)]  for all $p\in vE_A^*$ there exists some $q\in \source(p)E_A^*$ such that $B_{pq}=0$.
\end{enumerate}
\end{Theorem}

\begin{proof}
By Proposition \ref{prop:how_we_check_J=0}, $\hat J_K$ is non-zero exactly if there exists $v \in E_A^{0}$ such that $(\hat G_B)_v$ supports a singular element that is not tight. By Propositions \ref{Prop: faithful groups nonsingular}, \ref{Prop: nonfaithful2} and \ref{Prop: nonfaithful_3}, this happens exactly if  $(\hat G_B)_v$ intersects $\kappa \setminus \kappa_T$ and (T2) holds. By Lemma \ref{Lem: tight kernel characterization}, $(\hat G_B)_v$ intersects $\kappa \setminus \kappa_T$ exactly if $(\hat G_B)_v$ acts non-faithfully and there exists an infinite path $p \in vE_A^\omega$ such that no finite prefix $q$ satisfies $B_q=0$. We claim that this latter condition is equivalent to (T1). Indeed if (T1) holds, then $p=rc^\omega$ is clearly such an infinite path. Conversely, given any infinite path $p$, we can factorize it along its first repeated vertex as $p=rcw$ where $r$ is a simple path and $c$ is a simple cycle. If $B_q \neq 0$ for any finite prefix $q$ of $p$, then in particular $B_{rc} \neq 0$ as well. This completes the proof.
\end{proof}

We apply Theorem \ref{Thm:unfaithful J=0} to obtain our simplicity characterization for Katsura algebras.

\begin{Corollary}\label{Cor: nonfaithful simplicity main result}
Let $(A,B)$ be a Katsura pair and $(\hat G_B, E_A)$ the associated KEP-groupoid and let $K$ be any field. The Steinberg algebra $K\G_{(\hat G_B, E_A)}$ and the Katsura algebra $C^{\ast}_r(\G_{(\hat G_B, E_A)})$ are simple if and only if both of the following hold:

\begin{enumerate}
	\item[(C1)] $E_A$ has a unique nontrivial strongly connected component which contains at least two cycles;
	
	
	\item[(C2)] For any $v\in E_A^0$ with $(G_B)_v$ acting non-faithfully, $B_{pc}=0$ for any $pc\in vE_A^*$ where $p$ is a simple path and $c$ is a simple cycle.
	
\end{enumerate}
When these conditions hold, $C^{\ast}_r(\G_{(\hat G_B, E_A)})$ is purely infinite simple. Moreover, the simplicity of $K\G_{(\hat G_B, E_A)}$ does not depend on $K$.
\end{Corollary}
\begin{proof}
	Recall that the simplicity of $K\G_{(\hat G_B, E_A)}$ and $C^{\ast}_r(\G_{(\hat G_B, E_A)})$ coincide by Corollary \ref{Cor: cstar simplicity iff steinberg simplicity}, and if $C^{\ast}_r(\G_{(\hat G_B, E_A)})$ is simple, it is purely infinite simple by Proposition \ref{Prop: simple implies purely infinite}. We argue that these conditions are equivalent to the simplicity of $K\G_{(\hat G_B, E_A)}$.
	
	By Theorem \ref{Thm: ample groupoid simplicity characterisations}, $K\G_{(\hat G_B, E_A)}$ is simple if and only if $\G_{(\hat G_B, E_A)}$ is minimal and topologically free and the singular ideal vanishes. Minimality and topological freeness is equivalent to the conjunction of conditions (Cof), (Cyc), and (Evr). Because $E_A$ is finite and without sources, $(\hat G_B, E_A)$ satisfies (Cof) and (Cyc) if and only if (C1) is satisfied. 
	
	It follows from Theorem \ref{Thm:unfaithful J=0} that (C2) implies that the singular ideal vanishes, and from Lemma \ref{lem: fixed path computation} that (C2) implies condition (Evr) holds. Conversely, suppose $(\hat G_B, E_A)$ satisfies (Evr) and the singular ideal vanishes, and suppose $(\hat G_B)_v$ acts non-faithfully. By the definition of (Evr) and Lemma \ref{lem: fixed path computation}, there must be a zero path in $vE_A^*$. Observe that if $p\in vE_A^*$ is a non-zero path, the isotropy group $(\hat G_B)_{\source(p)}$ acts non-faithfully as well by Proposition \ref{Prop: Infinite group characterization}. In particular, there is a zero path in $\source(p)E_A^*$ so condition (T2) of Theorem \ref{Thm:unfaithful J=0} holds for $v$. The singular ideal vanishes, so condition (T1) of Theorem \ref{Thm:unfaithful J=0} must fail for $v$. It follows that (C2) holds.
	
	For the last statement, observe that (C1) and (C2) do not depend on $K$.
\end{proof}


\begin{Remark}
	\label{Rem: nonfaithful J=0 algorithm}
	
	We observe that on input $(A,B)$ we can effectively check the conditions of Theorem \ref{Thm:unfaithful J=0} and Corollary \ref{Cor: nonfaithful simplicity main result}. In particular, given a vertex $v \in E_A^0$, by Proposition \ref{Prop: Infinite group characterization} we can check if $(\hat G_B)_v$ acts faithfully. Condition (T1) can be verified by exhaustion.
	Condition (T2) is equivalent to checking that whenever $v$ is reachable from $u \in E_A^*$ via a non-zero path, $u$ is reachable from a zero edge. Condition (C2) can be treated similarly, and condition (C1) can be verified using typical graph algorithms.
\end{Remark}


\section{The singular ideal associated to faithful KEP-groupoids}
\label{Sec: faithful}

Throughout the section, let $(A,B)$ be a Katsura pair and let $K$ be any field. We turn our attention to the faithful KEP-groupoid $(G_B, E_A)$. We begin by stating a corollary of Proposition \ref{Prop: faithful groups nonsingular}.

\begin{Corollary}
\label{Cor: infinite groups nonsingular}
    Let $(G_B, E_A)$ be a faithful KEP-groupoid and $v \in E_A^{0}$. Suppose $(G_B)_v$ is infinite. Then $(G_B)_v$ does not support a non-zero singular element in $KS_{(G_B, E_A)}$.
\end{Corollary}

Let $G_{\text{fin}}$ be the subgroupoid of $G_B$ consisting of all finite subgroups of $G_B$. More precisely, $G_{\text{fin}}$ contains all finite isotropies of $G_B$ as well as the unit space. It follows from \cite[Cor. 3.3]{HUME_WHITTAKER_2025} that $G_{\text{fin}}$ is closed under taking sections, and thus the action of $G_{\text{fin}}$ on $E_A^*$ inherited from $G_B$ is self-similar. Observe furthermore by Propositions \ref{prop:tight_characterization} and \ref{prop:singular_characterization} that $a\in KS_{(G_{\text{fin}}, E_A)}$ is singular, respectively tight, in $KS_{(G_B, E_A)}$ exactly if
that $a$ is singular, respectively tight in $KS_{(G_{\text{fin}}, E_A)}$. Thus by Proposition \ref{prop:how_we_check_J=0} and Corollary \ref{Cor: infinite groups nonsingular}, $KS_{(G_{\text{fin}}, E_A)}$ contains singular and not tight elements exactly if $KS_{(G_B, E_A)}$ does.

Since $G_{\text{fin}}$ is finite, it is a \emph{contracting} self-similar groupoid (see \cite[Def. 3.3]{BrownloweEtAlKKDualSelfsimliarGroupoids}).
By Proposition \ref{Prop: Infinite group characterization}, on input $A, B$ we can compute the finite (and thus contracting) self-similar groupoid $(G_{\text{fin}}, E_A)$, therefore \cite[Cor. 8.7]{AAKRE2026130446} implies that on a further input of the characteristic of $K$, it is decidable whether $J_K=0$. This algorithm however runs in exponential time in the size of the nucleus, which itself can be exponential in the input $A,B$.
For the rest of the section, we work towards describing a more efficient algorithm, which in particular also yields that just like in the non-faithful case, the vanishing of $J_K$ is independent of the field $K$.

\subsection{Stable subgroups}

We recall that by \cite[Thm. 7.2]{AAKRE2026130446}, if $KS_{(G_{\text{fin}}, E_A)}$ contains a singular element which is not tight, then this can be assumed to be supported on a so-called \emph{recurring subgroup}.
Given a self-similar groupoid $(G,E)$, let $v \in E^0$ be a vertex and $p \in vE^*v$ be a cycle in $E$. The associated recurring subgroup of $G$ is
$$H_p=\{g \in (G)_v \colon g(p)=p, g|_p=g\}.$$
Thus the following proposition holds.

\begin{Proposition}
	\label{Prop: singular supported on recurring subgroups}
Suppose that $KS_{(G_B, E_A)}$ contains an element which is singular but not tight. Then there exists a finite isotropy $(G_B)_v$ and a cycle $p \in vE_A^*v$ such that $H_p$ supports a singular element which is not tight.
\end{Proposition}

We shall later need the following observations.

\begin{Lemma} 
	\label{Lem:recurring subgroups congjugate}
	Let $pq$ be a cycle in $E_A^*$ and consider the recurring subgroups $H_{pq}$ and $H_{qp}$. Then
	\begin{enumerate}
		\item The map $|_p: H_{pq} \to H_{qp}$ is a group isomorphism with inverse $|_q: H_{qp}\rightarrow H_{pq}$. 
		\item If $p, q$ are cycles themselves, and $(G_B)_{\source(p)}$ is finite then furthermore $H_{pq}=H_{qp}$.
	\end{enumerate}
\end{Lemma}

\begin{proof}
	For distinct $g,h\in H_{pq}$, we have $(gh)|_{p} = g|_{h(p)}h|_{p} = g|_ph|_p$ by \cite[Lemma 3.4]{LACA2018268}, so $|_p$ is a group homomorphism from $H_{pq}$. For each $g\in H_{pq}$, we have $g(pqp)=pqp$ so $g|_{p}(qp) = qp$. We also have $(g|_{p})|_{qp} = (g|_{pq})|_{p} = g|_{p}$, so $(H_{pq})|_{p}\leq H_{qp}$. Similarly, $|_q: H_{qp} \to H_{pq}$ is a group homomorphism.
	
	The maps are mutually inverse bijections, and hence isomorphisms, because for each $g\in H_{pq}$ and each $h\in H_{qp}$, he have $g=(g|_p)|_q$ and $h=(h|_q)|_p$.
	
	For (2), observe that $H_{pq}$ and $H_{qp}$ are isomorphic subgroups of the finite cyclic group $(G_B)_{\source(p)}$, and therefore $H_{pq}=H_{qp}$.
\end{proof}

By Proposition \ref{Prop: singular supported on recurring subgroups}, it suffices to find the recurring subgroups of finite isotropies, and search for singular elements here. A key advantage of this approach is that non-zero elements supported on recurring subgroups are automatically not tight.
Existing methods to find recurring subgroups are not computationally practical, the main issue being that not all recurring subgroups arise as recurring subgroups of simple cycles. Thus instead, we look at a larger class of subgroups we term \emph{stable}, which are easier to find but still guarantee non-tightness.

\begin{Lemma}
\label{Lem:restriction to c}
Suppose $v \in E_A^0$ is such that $(G_B)_v$ is finite.
Let $c$ be a cycle in $vE_A^*v$ and suppose $H \leq \Stab c \leq (G_B)_v$. Then $|_c$ defines a group homomorphism $|_c \colon H \to H, h \mapsto h|_c$.
\end{Lemma}

\begin{proof}
Let $(G_B)_v=\langle g_v \rangle$, $H=\langle g_v^m \rangle$, and suppose $|H|=M$ so $|(G_B)_v|=M m$. The section $|_c$ maps $(G_B)_v$ to $(G_B)_v$ and thus we can define a function $f \colon [0, Mm-1] \to [0, Mm-1]$ by $g_v^k|_c=g_v^{f(k)}$.

By Lemma \ref{lem: fixed path computation}
we have
\begin{equation}
	\label{eq: c-morhpism formula}
f(k) \equiv k \frac{B_c}{A_c} \pmod{Mm}.
\end{equation}
Since $(G_B)_v$ is finite, we must have $\frac{B_c}{A_c} \in \mathbb Z$ by Proposition \ref{Prop: Infinite group characterization}. So if $m$ divides $n$, it also divides $f(n)$, which shows that $H|_c \subseteq H$.

To show that $|_c$ is a group homomorphism, observe that for any $h_1, h_2 \in H$, we have 
\[(h_1h_2)|_c = h_1|_{h_2(c)}h|_c = h_1|_ch_2|_c\]
by \cite[Lemma 3.4]{LACA2018268}.
\end{proof}

Observe by (\ref{eq: c-morhpism formula}) that $f(km)=0$ exactly if $Mm$ divides $km \frac{B_c}{A_c}$, so the kernel of $|_c \colon H \to H$ is generated by
$(g_v^{m})^k$ where 
\begin{equation}
\label{eqn:order-of-k}
k=\frac{M}{ \gcd\left(M, \frac{B_c}{A_c}\right)}
\end{equation}
and $M=|H|$.
Moreover, $mk$ is the minimal exponent of $g_v$ which generates the kernel.

\begin{Definition}\label{Def: c-stable subgroups}
	Let $(G_B,E_A)$ be a faithful KEP-groupoid, and suppose $v \in E_A^{0}$ is such that $(G_B)_v$ is finite. Let $H \leq (G_B)_v$ be a subgroup and $c \in  vE_A^*v$ a cycle. We call $H$ \emph{$c$-stable} if $H$ stabilizes $c$, and the homomorphism $|_c \colon H \to H$ defined by $h \mapsto h|_c$ is a bijective.
\end{Definition}
Observe that if $B_c=0$, then the only $c$-stable subgroup is the trivial group.

\begin{Lemma}
\label{Lem:cstable} 
Let $v \in E_A^0$ be such that $(G_B)_v$ is finite, and let $c \in vE_A^*v$ be a cycle. 
Then $(G_B)_v$ has a unique maximal $c$-stable subgroup, which we denote by $G_c$.
\end{Lemma}

\begin{proof}
Consider the morphism  $|_c \colon \Stab c \to \Stab c$ from Lemma \ref{Lem:restriction to c} for $H=\Stab c$, and let $\ker(|_c)$ denote its kernel.
By definition, a subgroup $G \leq \Stab c \leq (G_B)_v$ is $c$-stable exactly if $G \cap \ker(|_c) =v$.
Following the notation of Lemma \ref{Lem:restriction to c} and the following equation (\ref{eqn:order-of-k}), put $\Stab c=\langle g\rangle$ where $g=g_v^r$, let $M:=|\Stab c|,$ suppose $G=\langle g^n \rangle$ and $\ker(|_c)=\langle g^k \rangle $, where $n$ and $k$ are assumed to be minimal.
Then $G$ is $c$-stable exactly if $\lcm(k,n)=M$.
If $M$ and $k$ have prime factorizations 
$$M=\prod_{i \in \mathbb N} p_i^{\alpha_i} \hbox{ and } k=\prod_{i \in \mathbb N} p_i^{\beta_i},$$
then $\lcm(k,n)=M$ exactly if 
$n=\prod p_i^{\gamma_i}$ with $\max(\beta_i,\gamma_i)=\alpha_i$. There is a unique minimal (with respect to divisibility) such $n$ with this property, namely given by 
\begin{equation}\label{eq: max c-stable generator}
	\gamma_i=
	\begin{cases}
		0 &\hbox{if } \beta_i=\alpha_i\\
		\alpha_i & \hbox{otherwise}.
	\end{cases}
\end{equation}
The corresponding subgroup $G_c:=\langle g^n \rangle$ is the unique maximal $c$-stable subgroup of $(G_B)_v$.
\end{proof}
Our next goal is to show that $|G_c|$ is computable on input $A,B$ and a cycle $c$. The missing piece is computing $|\Stab c|$: 

\begin{Lemma}\label{lem:stabilizer size}
	Suppose $|(G_B)_v|=N$ and $c \in vE_A^*v$ is such that $B_c \neq 0$. Then $|\Stab c|=\frac{N}{r}$ where
	$$r=\lcm\left\{ \frac{A_{c'}}{\gcd(A_{c'}, |B_{c'}|)} \colon c' \hbox{ is a prefix of } c\right\}.$$
\end{Lemma}

\begin{proof}
	By Lemma \ref{Lem:fixed path}, we have $g_v^k \in \Stab c$ exactly if $\frac{A_{c'}}{\gcd(A_{c'}, |B_{c'}|)}$ divides $k$ for all prefixes $c'$ of $c$. The claim follows.
\end{proof}

\begin{Remark}
	\label{Rem: G_c computable}
	It follows that for a vertex $v\in E_A^0$ with $N=|(G_B)_v|$ and a cycle $c\in vE_A^*v$ with $B_c \neq 0$, we can compute the order of $G_c$ by applying Lemma \ref{lem:stabilizer size} to obtain $M=|\Stab c|$ and $r=\frac{N}{M}$, so that $\Stab c=\langle g_v^r \rangle$. We then apply (\ref{eqn:order-of-k}) for $H=\Stab c$ to compute $k$, 
	and (\ref{eq: max c-stable generator}) to obtain the number $n$ such that $G_c=\langle g_v^{rn} \rangle$. The order of $G_c$ is then $\frac{N}{nr}=\frac{M}{n}$. We remark that one does not need to prime factorize to compute $n$, as we shall see in Section \ref{Sec: complexity analysis}.
\end{Remark}

Observe that recurring subgroups are, by definition, stable. Thus Proposition \ref{Prop: singular supported on recurring subgroups} also implies that it suffices to search for singular elements on maximal stable subgroups (we will later see that it suffices to consider those which belong to simple cycles).

As we mentioned, non-zero singular elements supported on stable subgroups are never tight:

\begin{Proposition}
	\label{Prop: c-stable never tight}
Let $H$ be a $c$-stable subgroup, and let $0\neq a\in KH$. Then $a \notin T_K$.
\end{Proposition}

\begin{proof}
To show that $a \notin \TightIdeal$ it suffices by Proposition \ref{prop:tight_characterization} that $ac^k \neq 0$ for any $k \geq 0$. Since $H$ is $c$-stable, putting $a=\sum_{h \in H} a_h h$, we have $ac^k=c^k\sum_{h \in H} a_h f^k(h)$ where $f \colon H \to H$ is the bijection defined by $f(h)=h|_c$. In particular $f^k$ is also a bijection so $0<|\supp a|=|\supp (ac^k)|$ which completes the proof.
\end{proof}

\subsection{Finding singular elements in subgroups.}
We now turn to investigate how we can check if there are nonzero singular elements supported on given subgroups, which we will apply to maximal stable subgroups associated to simple cycles.

We first show a converse to Proposition \ref{Prop: free orbits prevent singular elements} and prove that the presence of non-zero singular elements in some subgroup $H$ is entirely determined by whether $H$ has free orbits.

For a self-similar groupoid $(G, E)$ and some $g\in G$, an infinite path $p\in \source(g)E^\omega$ is called $g$-\textit{generic} if $g(p)\neq p$ or $g(p)=p$ and $g|_q = \source(q)$ for some prefix $q$ of $p$. We say that $p\in E^\omega$ is \textit{generic} if $p$ is $g$-generic for all $g\in (G)_{\range(p)}$. We also say a finite path $q\in E^*$ is generic if every path in $qE^\omega$ is generic.
In \cite{Nek2009CstarAlg}, Nekrashevych observes a more general form of the following for countable self-similar groups acting on a finite alphabet. We provide the following proof for completeness.
\begin{Lemma}\label{Lem: there exists a generic finite path}
    Suppose $(G_B)_v$ is finite. For all $p\in vE_A^*$, there exists some $q\in \source(p)E_A^*$ such that $pq$ is generic.
\end{Lemma}
\begin{proof}
    Let $p\in vE_A^*$. By definition, every path in $pE_A^\omega$ is already $h$-generic for any $h\in (G_B)_v$ such that $h(p)\neq p$. Suppose that $g\in (G_B)_v$ and $g(p)=p$. If $g|_p= \source(p)$, every path in $pE_A^\omega$ is $g$-generic. If $g|_p\neq \source(p)$, then by the faithfulness of the action, there exists some $pq\in vE_A^*$ such that $g(pq)\neq pq$. But then every path in $pqE_A^\omega$ is $g$-generic. Because $(G_B)_v$ is finite, we can perform finitely many extensions to $p$ to attain a generic path $pq\in vE_A^*$.
\end{proof}

The following allows us to write down a concrete singular element in any subgroup without free orbits.

\begin{Proposition}\label{Prop: pi(g) is singular and non-zero}
   Let $(G_B, E_A)$ be a faithful KEP-groupoid and $K$ be a field. Suppose $(G_B)_v$ is finite and $H$ is a nontrivial subgroup of $(G_B)_v$ generated by $g$ with no free orbits. Put $m=|H|$. Then
    \[\pi(g):=\prod_{1\leq d < m, d|m}(v-g^d) \in KH\]
    is a non-zero singular element of $KS_{(G_B, E_A)}$.
\end{Proposition}

\begin{proof}
To see that $\pi(g)$ is singular, we apply Proposition \ref{prop:singular_characterization}. If $p \notin vE_A^*$, then $\pi(g)p=0$. Assume $p\in vE_A^*$. By Lemma \ref{Lem: there exists a generic finite path}, there exists some $q\in \source(p)E_A^*$ such that $pq$ is generic. Let $w\in pqE_A^\omega$, and consider the orbit of $w$ under the action of $H$. The orbit of $w$ is not free by assumption, so it has size $d$ for some $1\leq d<m$ where $d$ divides $m$. Then $g^d(w)=w$ implying that $g^d(pq)=pq$. Because $pq$ is generic, $g^d|_{pq}=\source(pq)$, and so $(v-g^d)pq = 0$. Observe that $(v-g^d)$ is a factor of $\pi(g)$; we conclude that $\pi(g)pq = 0$ because $H$ is abelian. This confirms that $\pi(g)$ is singular.

We move on to prove that $\pi(g) \neq 0$.
    Observe that $KH\cong \faktor{K[x]}{(x^m-1)}$, and $\pi(g)$ is the image of $\pi(x):= \prod_{1\leq d < m, d|m}(1-x^d)$ under the natural projection from $K[x]$ to $KH$. Then $\pi(g)\neq 0$ if and only if $\pi(x)$ is not divisible by $(x^m-1)$. Without loss of generality, assume that $K$ is algebraically closed. We have 
    $$x^m-1 = \prod_{\zeta\in K, \zeta^m=1}(x-\zeta).$$
    We proceed in two cases. First, suppose the characteristic of $K$ does not divide $m$. Then there exists some $\zeta\in K$ with multiplicative order $\sigma(\zeta)=m$. It follows that $\zeta$ is a root of $x^m-1$, but not a root of $\pi(x)$, so $x^m-1$ does not divide $\pi(x)$. Now suppose that the characteristic of $K$ is $p$ with $m=p^k\cdot n$ with $p$ not dividing $n$ and $k\geq 1$. Then $x^m-1 = x^{{p^k}n}-1=(x^n-1)^{p^k}$. There exists some $\zeta\in K$ with $\sigma(\zeta)=n$. We have that $\zeta$ is a root of $x^m-1$ with multiplicity $p^k$. Additionally, $\zeta$ is a root of $\pi(x)$ with multiplicity 
    \[\left|\{1\leq d< m \colon n|d, d|p^kn\}\right|=\left|\{n,np,\ldots, np^{k-1}\}\right|=k.\]
    Since $p\geq 2$ and $k\geq 1$, we have $p^k>k$ so again $x^m-1$ does not divide $p(x)$.

    In both cases, we find that $\pi(g)$ is non-zero in $KH$.
\end{proof}

An immediate consequence of Propositions \ref{Prop: free orbits prevent singular elements} and \ref{Prop: pi(g) is singular and non-zero} is the following.

\begin{Corollary}\label{Cor: Subgroups without free orbits support singular elements}
    Let $(G_B, E_A)$ be a faithful KEP-groupoid and $v \in E_A^0$. Suppose $(G_B)_v$ is finite and $H \leq (G_B)_v$. Then $H$ supports a non-zero singular element if and only if the action of $H$ on $vE_A^\omega$ has no free orbits.
\end{Corollary}

If $H$ is a stable subgroup, by Proposition \ref{Prop: c-stable never tight} a non-zero singular element is automatically not tight. As we have already observed, it suffices to search for singular elements in stable subgroups -- we will now see that we can further assume they belong to simple cycles.

\begin{Proposition}
	\label{Prop: c-stable simple c suffices}
Suppose $KS_{(G_B, E_A)}$ contains a singular element which is not tight. Then there is a vertex $v \in E_A^0$ and a simple cycle $c \in vE_A^*v$ with $B_c \neq 0$ such that the maximal $c$-stable subgroup contains a singular element which is not tight.
\end{Proposition}

\begin{proof}
Suppose $KS_{(G_B, E_A)}$ contains a singular element which is not tight. By Corollary \ref{Cor: Subgroups without free orbits support singular elements} it suffices to show that there is a $c$-stable subgroup $H$ with no free orbits, where $c$ is a simple cycle. Necessarily we must have $B_c \neq 0$.
 By Proposition \ref{Prop: singular supported on recurring subgroups}, there is a cycle $p \in E_A^*$ such that the recurring subgroup $H_p$ supports such an element. In particular the action of $H_p$ has no free orbits by Proposition \ref{Prop: free orbits prevent singular elements}. If $p$ happens to be a simple cycle, then we may put $v=\range(p)$, $c=p$ and $H=H_p$ -- it follows by definition that $H_p$ is $c$-stable.

If $p$ is not simple, then let us factor $p$ at the first repeated vertex as $p=p_1 c p_2$ where $p_1$ is a simple path and $c$ is a simple cycle. We may apply Lemma \ref{Lem:recurring subgroups congjugate} (1) to $p:=p_1, q:=cp_2$ to obtain an isomorphism from $H_p$ to $H_{c p_2p_1}$. Furthermore if $H_{c p_2p_1}$ had a free orbit, i.e. if some path $q \in r(c)E_A^\omega$
had $H_{c p_2p_1}$-orbit of size $|H_p|$, then the $H_p$-orbit of $p_1q \in r(p)E_A^\omega$ would have the same size and would thus be free. So $H_{c p_2p_1}$ has no free orbit. We claim that it is also $c$-stable. It is clear that $H_{c p_2p_1}$ stabilizes $c$. Applying Lemma \ref{Lem:recurring subgroups congjugate} (1) to $p:=c$ and $q:=p_1p_2$, we have that  $|_c \colon H_{c p_2p_1} \to H_{p_2p_1 c}$ is a bijection, but by  Lemma \ref{Lem:recurring subgroups congjugate} (2) $H_{c p_2p_1} = H_{p_2p_1 c}$ which shows that $H=H_{c p_2p_1}$ is a $c$-stable subgroup with no free orbits.

In either case, the maximal $c$-stable subgroup $G_c$ supports a singular element which is not tight.
\end{proof}

We state the main theorem of the section.

\begin{Theorem}
\label{Thm: testing singular elements faithful}
Let $(G_B, E_A)$ be a KEP-groupoid.
Then $K\G_{(G_B, E_A)}$ has a nontrivial singular ideal
if and only if there is a vertex $v \in E_A^0$ with $(G_B)_v$ finite and a simple cycle $c \in vE_A^*v$ such that the unique maximal $c$-stable subgroup $G_c$ has no free orbits on $vE_A^\omega$.

In this case, an element of $KG_c$ which is singular and not tight is
$$a=\prod_{1\leq d<m, d\mid m} (v-g_v^{dn})$$
where $|H|=m$, $(G_B)_v=\langle g_v \rangle$, $H=\langle g_v^n \rangle$.
\end{Theorem}
\begin{proof}
The first statement follows from Propositions \ref{prop:how_we_check_J=0}, \ref{Prop: c-stable simple c suffices} and Corollary \ref{Cor: Subgroups without free orbits support singular elements}. 
	
Furthermore, if $G_c$ has no free orbits, then by Proposition \ref{Prop: pi(g) is singular and non-zero},
$a \in KG_c \cap \SingularIdeal \setminus \{0\}.$
It follows from Proposition \ref{Prop: c-stable never tight} that $a \notin \TightIdeal$.
\end{proof}

\begin{Corollary}
\label{Cor: faithful simplicity main result}
	Let $(A,B)$ be a Katsura pair and $( G_B, E_A)$ the associated faithful KEP-groupoid and let $K$ be any field. The Steinberg algebra $K\G_{( G_B, E_A)}$ and the Katsura algebra $C^{\ast}_r(\G_{( G_B, E_A)})$ are simple if and only if the following hold:
	
	\begin{enumerate}
		\item[(C1)] $E_A$ has a unique nontrivial strongly connected component which contains at least two cycles;
		
		\item[(FC2)] For any vertex $v\in E_A^0$ with $(G_B)_v$ finite and any simple cycle $c \in vE_A^*v$, the unique maximal $c$-stable subgroup $G_c$ has a free orbit on $vE_A^\omega$.
		
	\end{enumerate}
	When these conditions hold, $C^{\ast}_r(\G_{( G_B, E_A)})$ is purely infinite simple. Moreover, the simplicity of $K\G_{( G_B, E_A)}$ does not depend on $K$.
\end{Corollary}
\begin{proof}
	Recall that the simplicity of $K\G_{( G_B, E_A)}$ and $C^{\ast}_r(\G_{( G_B, E_A)})$ coincide by Corollary \ref{Cor: cstar simplicity iff steinberg simplicity}, and if $C^{\ast}_r(\G_{( G_B, E_A)})$ is simple, it is purely infinite simple by Proposition \ref{Prop: simple implies purely infinite}. We argue that these conditions are equivalent to the simplicity of $K\G_{( G_B, E_A)}$.
	
	By Theorem \ref{Thm: ample groupoid simplicity characterisations}, $K\G_{( G_B, E_A)}$ is simple if and only if $\G_{( G_B, E_A)}$ is minimal and topologically free and the singular ideal vanishes. Minimality and topological freeness is equivalent to the conjunction of conditions (Cof), (Cyc), and (Evr). Because $E_A$ is finite and without sources, $( G_B, E_A)$ satisfies (Cof) and (Cyc) if and only if (C1) is satisfied, and (Evr) is always satisfied for faithful actions.

	It is clear that (FC2) is equivalent to the singular ideal vanishing by Theorem \ref{Thm: testing singular elements faithful}. For the last statement, observe that (C1) and (FC2) do not depend on $K$.
\end{proof}

%
%

We also observe:

\begin{Corollary}
	For a Katsura pair $(A, B)$, $C^{\ast}_r(\G_{( G_B, E_A)})$ is simple if and only if the graph algebra $C^{\ast}_r(E_A)$ is simple and the singular ideal $J$ vanishes.
\end{Corollary}
\begin{proof}
	Observe that the graph algebra $C^{\ast}_r(E_A)$ is exactly $C^{\ast}_r(\G_{( G_B^0, E_A)})$. Both $( G_B, E_A)$ and $( G_B^0, E_A)$ act faithfully, so minimality and topological freeness of $\G_{( G_B^0, E_A)}$ and $\G_{( G_B, E_A)}$ are characterized by (Cof) and (Cyc) which depend only on $E_A$.
\end{proof}

\begin{Remark}
Recall that by Proposition \ref{Prop: simple implies purely infinite},  if	$C^{\ast}_r(\G_{( G_B, E_A)})$ is simple, it is purely infinite simple, and thus in particular a UCT Kirchberg algebra.
\end{Remark}

By Theorem \ref{Thm: testing singular elements faithful}, we can effectively decide if $J_K=0$ if we can decide for any simple cycle $c$ whether $G_c$ has a free orbit. The following allows us to do that.

\begin{Proposition}\label{Prop: Subgroups with free orbits characterization}
	Let $(G_B, E_A)$ be a KEP-groupoid and $v \in E_A^0$ such that $(G_B)_v$ is finite. Suppose that the subgroup $\langle g \rangle =H \leq (G_B)_v$ has order $m >1$. The following are equivalent.
	\begin{enumerate}
		\item The action of $H$ on $vE_A^\omega$ has a free orbit;
		\item There exists a simple path $w\in vE_A^*$ such that $i\frac{B_{w}}{A_{w}}\notin \ZZ$ for all $1\leq i< m$;
		\item The action of $H$ on the simple paths of $vE_A^*$ has a free orbit.      
	\end{enumerate}
\end{Proposition}
\begin{proof}
	Beginning with $(1)$, suppose that the orbit of $p\in vE_A^\omega$ under the action of $H$ is free. Let $q$ be the shortest prefix of $p$ such that $g^{i}(q)\neq q$ for all $1\leq i< m$. It follows from Lemma \ref{Lem:fixed path} that $i\frac{B_q}{A_q}\notin \ZZ$ for all $1\leq i< m$. If $q$ is not a simple path, there exists cycles $c_j$ such that $q=w_1c_1w_2c_2\ldots w_n$ and $w:= w_1\ldots w_n$ is a simple path. Recall that by Proposition \ref{Prop: Infinite group characterization}, $\frac{B_{c_j}}{A_{c_j}}\in \ZZ$. 
	Thus, $w\in vE_A^*$ and $i\frac{B_{w}}{A_{w}}\notin \ZZ$ for all $1\leq i< m$, establishing $(2)$. Finally, applying Lemma \ref{Lem:fixed path}, we get that the orbit of $w$ under the action of $H$ is free, establishing $(3)$. It is immediate that $(3) \implies (1)$.
\end{proof}

\begin{Corollary} 
	\label{Cor: H supports singular characterization}
	Let $(G_B, E_A)$ be a faithful KEP-groupoid. Suppose that $H\leq (G_B)_v$ has finite order $m>1$. Then $H$ supports a non-zero singular element if and only if for every simple path $q\in vE_A^*$ with $B_q\neq 0$ we have
	$$\frac{A_q}{\gcd(A_q, |B_q|)}<m.$$
\end{Corollary}
\begin{proof}
	By Corollary \ref{Cor: Subgroups without free orbits support singular elements}, $H$ supports a non-zero singular element if and only if the action of $H$ on $E_A^{\omega}$ has no free orbits. By Proposition \ref{Prop: Subgroups with free orbits characterization}, this occurs if and only if for all simple paths $q\in vE_A^*$ there exists some $1\leq i < m$ such that $i\frac{B_q}{A_q}\in \ZZ$ or equivalently, for all simple paths $q\in vE_A^*$ with $B_q\neq 0$ we have that $\frac{A_q}{\gcd(A_q, |B_q|)}<m$.
\end{proof}

Recall that on input $A, B$ and a simple cycle $c$, one can compute $|(G_B)_{\source(c)}|$ by Proposition \ref{Prop: Infinite group characterization} and thus $|G_c|$ by Remark \ref{Rem: G_c computable}. By Corollary \ref{Cor: H supports singular characterization} we can thus decide if $G_c$ supports any singular elements. In particular, by Theorem \ref{Thm: testing singular elements faithful} we can decide if $J_K=0$. We will see in Section \ref{Sec: complexity analysis} that this algorithm is in polynomial space.


\section{Space complexity analysis}

\label{Sec: complexity analysis}

In this section we analyze the space complexity of the algorithms described in Sections \ref{Sec: nonfaithful} and \ref{Sec: faithful}. Let $(A, B)$ be a Katsura pair, recalling that $A$ has no zero rows. In particular, we show that on input $(A,B)$, we can decide in polynomial space if the singular ideal of $K\G_{(\hat G_B, E_A)}$, respectively $K\G_{(G_B, E_A)}$ vanish.

By a polynomial space algorithm, we mean a (deterministic) Turing machine which runs in space bounded by some polynomial in the size of the input. (We will describe the effective procedure executed by the Turing machine.)

Throughout, let $\mathfrak m$ denote the largest entry of $A$ and $B$, and let $\mathfrak N$ denote the common dimension (i.e. number of rows) of $A$ and $B$.
We always assume that integers are encoded in binary, in particular the size of the input $(A,B)$ is logarithmic in $\mathfrak m$ and quadratic in $\mathfrak N$.

It is helpful to define the quotient of $E_A$ where we collapse all multiedges to a single edge, which we call the \emph{underlying simple graph} of $E_A$. We denote this by $\overline{E}_A$. Note that since $A_e$ and $B_e$ only depend on the source and range of $e$, they remain well-defined for edges and paths of $\overline{E}_A$, which we tacitly use. Crucially, $c$-stable subgroups depend only on the vertex sequence of $c$.

Recall that we identify the number of vertices of $E_A$ (and thus $\overline{E}_A$) by the set $[1, \mathfrak N]$. In particular, paths in the simple graph $\overline{E}_A$ can be naturally encoded as a sequence of vertices, and they are naturally ordered by lexicographic order. In the algorithms presented below, this is the encoding of paths we use, and whenever we cycle through all vertices and simple paths or cycles, we do so in lexicographic order, that is, by incrementation.

We will regularly use that a simple path or cycle in $\overline E_A^*$ traverses at most $\mathfrak N+1$ vertices, each of which is a number between $1$ and $\mathfrak N$. The path or cycle is uniquely determined by this sequence which can be stored in $O(\mathfrak N \log \mathfrak N )$ space; this is the encoding we use. Similarly, if $p$ is a simple path or cycle in $\overline E_A^*$ then $A_p$ and $B_p$ are bounded above by
$\mathfrak m^{\mathfrak N}$ and can be stored in $O(\mathfrak N\log \mathfrak m)$ space. All of these are polynomial in the size of the input, which has size at least $\mathfrak N+\log \mathfrak m$.

It is well-known that basic arithmetic operations, as well as the gcd and lcm can be computed in polynomial time (and thus polynomial space) and therefore we disregard the space these computations use.

As an important first step in understanding both faithful and non-faithful KEP-groupoids, we apply Proposition \ref{Prop: Infinite group characterization} to determine which isotropy groups act faithfully. Algorithm \ref{Alg: decide faithfulness} takes a Katsura pair $(A, B)$ and a vertex $v\in E_A^0$ as input and returns TRUE if the isotropy group $(\hat G_B)_v$ acts faithfully, and FALSE otherwise.
\begin{algorithm}
	\caption{Determining whether $(\hat G_B)_v$ acts faithfully}\label{Alg: decide faithfulness}
\begin{algorithmic}
	\Procedure{Faithful}{$A,B,v$}
	\ForAll{$p \in v\overline {E}_A^*$ simple paths with $B_p \neq 0$}
	\ForAll{$c \in \source(p)\overline{E}_A^*\source(p)$ simple cycles}
	\If{$\frac{B_c}{A_c} \notin \mathbb Z$} halt and \textbf{return} TRUE
	\EndIf
	\EndFor
	\EndFor
	\State \textbf{return} FALSE
	\EndProcedure
\end{algorithmic}
\end{algorithm}
\begin{Proposition}
	Algorithm \ref{Alg: decide faithfulness} runs in polynomial space of the size of $(A, B)$.
\end{Proposition}
\begin{proof}
	At any given time, what we need to store (in addition to $(A,B)$) is:
	\begin{itemize}
		\item the vertex $v$, which has size $O(\mathfrak N)$;
		\item the current path $p$ and cycle $c$, which have size $O(\mathfrak N \log \mathfrak N )$;
		\item  the variables $B_p, B_c$ and $A_c$, which have size $O(\mathfrak N\log \mathfrak m)$.
	\end{itemize}
	The claim follows.
\end{proof}
\subsection{KEP-groupoids}
\label{Sec: complexity analysis-faithful}
Denote by $\hat I_K$ and $\hat T_K$ the singular and tight ideals, respectively, of $KS_{(\hat G_B, E_A)}$. Recall that by Theorem \ref{Thm:unfaithful J=0}, an isotropy group $(\hat G_B)_v$ supports singular elements which are not tight if and only if the action of $(\hat G_B)_v$ is non-faithful and
\begin{enumerate}
	\item[(T1)]  there exist a simple path $r \in vE_A^*$ and a simple cycle $c \in \source(r) E_A^* \source(r)$ with $B_{rc} \neq 0$; and
	\item[(T2)]  for all $p\in vE_A^*$ there exists some $q\in \source(p)E_A^*$ such that $B_{pq}=0$.
\end{enumerate} Algorithm \ref{alg: decide simplicity unfaithful} decides the above conditions as outlined in Remark \ref{Rem: nonfaithful J=0 algorithm}, in particular it takes a Katsura pair $(A,B)$ as input and returns YES if $\hat I_K \setminus \hat T_K=\emptyset$ in $KS_{(\hat G_B, E_A)}$ and NO otherwise.
\begin{algorithm}
	\caption{Deciding if $\hat I_K \setminus \hat T_K=\emptyset$ for $(\hat G_B, E_A)$}\label{alg: decide simplicity unfaithful}
	\begin{algorithmic}
		\ForAll{$v \in {E}_A^0$}
			\If{\textsc{Faithful}$(A, B, v)=\text{ FALSE}$}\Comment{Checking condition (T1)}
				\ForAll{$r \in v\overline {E}_A^*$ simple paths with $B_r \neq 0$}
				\ForAll{$c \in \source(r)\overline {E}_A^*\source(r)$ simple cycles}
					\If{$B_c \neq 0$} T$1 \gets$ TRUE 
					\EndIf
				\EndFor
				\EndFor 
			\bigskip
				\If{T$1 = $ TRUE}	\Comment{Checking condition (T2)}
					\ForAll{$p \in v\overline {E}_A^*$ simple paths with $B_p \neq 0$}
					\ForAll{$q \in \source(p)\overline {E}_A^*$ simple paths}
					\If{$B_{q} = 0$} 
					\State \textbf{continue} with next $p$
					\EndIf
					\EndFor
					\State \textbf{continue} with next $v$
					\EndFor
					\State halt and \textbf{return} NO \Comment{We found $v$ satisfying Theorem \ref{Thm:unfaithful J=0}}
				\EndIf
			\EndIf
		\EndFor
		\State \textbf{return} YES
	\end{algorithmic}
\end{algorithm}

\begin{Proposition}
	\label{Prop: Alg1 poly space}
	
	Algorithm \ref{alg: decide simplicity unfaithful} runs in polynomial space of the size of $(A,B)$.
\end{Proposition}

\begin{proof}
	
	At any given time, what we need to store (in addition to $(A,B)$) is:
	\begin{itemize}
		\item the vertex $v$, which has size $O(\mathfrak N)$;
		\item the current paths $r,c, p$ and $q$, which have size $O(\mathfrak N \log \mathfrak N )$;
		\item  the variables $B_r, B_c, B_p$ and $B_q$, which have size $O(\mathfrak N\log \mathfrak m)$.
	\end{itemize}
	The claim follows.
	
	\begin{Remark}\label{Rem: nonfaithful simplicity algorithm}
		Recall that the simplicity of $K\G_{(\hat G_B, E_A)}$ and $C^{\ast}_r(\G_{(\hat G_B, E_A)})$ coincide by Corollary \ref{Cor: cstar simplicity iff steinberg simplicity}. Algorithm \ref{alg: decide simplicity unfaithful} can be extended to a polynomial space algorithm to decide simplicity of $K\G_{(\hat G_B, E_A)}$ by checking whether $\G_{(\hat G_B, E_A)}$ is minimal and topologically free.
		
		As discussed in Section \ref{Subsec: semigroup and groupoids algebras}, minimality is equivalent to condition (Cof) and topological freeness is equivalent to the conjunction of conditions (Evr) and (Cyc).
		
		Condition (Cof) is satisfied if $E_A$ has a unique nontrivial strongly connected component. This can be checked in polynomial time (and therefore polynomial space) by calculating $A^{\mathfrak N}$ (where $\mathfrak N$ is the dimension of $A$). Observe that zero columns of $A^{\mathfrak N}$ correspond to vertices from which a cycle is not reachable. Therefore condition (Cof) is equivalent to irreducibility of the square matrix corresponding to the vertices with non-zero columns in $A^{\mathfrak N}$. When the graph satisfies (Cof), condition (Cyc) is equivalent to the unique nontrivial strongly connected component having at least one more edge than vertex, which can be checked in polynomial time.
		
		Condition (Evr) is satisfied if whenever $g\in \hat G_B$ acts trivially on $\source(g)E_A^*$, there is some $p\in \source(g)E_A^*$ such that $g|_p=\source(p)$. By Lemma \ref{lem: fixed path computation}, this occurs if and only if whenever the action of $(G_B)_v$ fails to be faithful, there is some $p\in vE_A^*$ with $B_p=0$. Clearly $p$ may be chosen to be a simple path or cycle, so condition (Evr) can be checked in polynomial space.
	\end{Remark}
\end{proof}
\subsection{Faithful KEP-groupoids}
\label{Sec: complexity analysis-unfaithful}


Denote by $ I_K$ and $ T_K$ the singular and tight ideals, respectively, of $KS_{( G_B, E_A)}$. To decide if $\SingularNotTightIdeal=\emptyset$, we follow Theorem \ref{Thm: testing singular elements faithful}. More precisely, for every vertex $v$ such that $(G_B)_v$ is finite, and every simple cycle $c \in v\overline{E}_A^*v$, we compute the maximal $c$-stable subgroup $G_c$, and then need to determine if it has any free orbits using Proposition \ref{Prop: Subgroups with free orbits characterization} (2). 

Recall that for any vertex $v$, we can compute the order of $(G_B)_v$ using Proposition \ref{Prop: Infinite group characterization}. The procedure coded below returns the order of $(G_B)_v$ when it is finite, and $0$ otherwise. 

\begin{algorithm}
	\caption{Computing the order of $(G_B)_v$}\label{alg:order of isotropies}
	\begin{algorithmic}
		\Procedure{Order}{$A,B,v$}
		\If{\textsc{Faithful}$(A, B, v)=\text{ TRUE}$}
			\State \textbf{return} $0$
		\Else 
		\State $N \gets 1$
			\For{$p \in v\overline {E}_A^*$ simple path with $B_p \neq 0$}
				\State $N \gets \lcm\left(N, \frac{A_p}{\gcd(Ap,|B_p|)}\right)$		
			\EndFor
		
		\State \textbf{return} $N$
		\EndIf
		\EndProcedure
	\end{algorithmic}
\end{algorithm}

\begin{Proposition}
	\label{Prop: Alg2 poly space}
	
	Algorithm \ref{alg:order of isotropies} runs in polynomial space of the size of $(A,B)$, in particular, the output $N$ has size polynomial in the size of $(A,B)$.
\end{Proposition}

\begin{proof}

 At any given time, what we need to store (in addition to $(A,B)$) is:
	\begin{itemize}
		\item the vertex $v$, which has size $O(\mathfrak N)$;
		\item the current path $p$ which has size $O(\mathfrak N \log \mathfrak N )$;
		\item  the variables $A_p, B_p$, which have size $O(\mathfrak N\log \mathfrak m)$;
		\item the variable $N$ is bounded by $\prod_{e \in \overline{E}_A} A_e \leq \mathfrak m ^{{\mathfrak N}^2}$ so its size is $O({\mathfrak N}^2\log \mathfrak m)$.
	\end{itemize}
The claim follows.
\end{proof}

Next we consider the procedure outlined in Remark \ref{Rem: G_c computable}, which computes the order of the maximal $c$-stable subgroup of some cycle $c$ with $B_c \neq 0$, given the order $N$ of $(G_B)_v$. 

%

It is not immediately obvious, but computing the number $n$ as defined in (\ref{eq: max c-stable generator}) does not require prime factorization due to the following.

\begin{Lemma}
Let $M,k \in \mathbb Z^+$ and suppose $k \mid M$. Let $s_0=k$, and put 
$$s_{i+1}=\frac{s_i}{\gcd(s_i, \frac{M}{s_i})}.$$
Then $s_i$ is a decreasing, positive sequence of integers which eventually stabilizes at some value $s$. Furthermore, $n=\frac{M}{s}$ is the least integer with $\lcm(n,k)=M$.
\end{Lemma}

\begin{proof}
Recall from Lemma \ref{Lem:cstable} that given prime factorizations $M=\prod p_i^{\alpha_i}$ and
$k=\prod p_i^{\beta_i}$,
then least integer with $\lcm(n,k)=M$ is $n=\prod p_i^{\gamma_i}$ where
$$\gamma_i=
\begin{cases}
0 &\hbox{if } \beta_i=\alpha_i\\
\alpha_i & \hbox{otherwise}.
\end{cases}$$ 

Fix a prime $p$, and let $\alpha$, $\beta$ and $\delta_i$ denote the exponents of $p$ in the prime decompositions of $M$, $k$ and $s_i$ respectively. Note that $\alpha \geq \beta$.
By definition we have $\delta_0=\beta$ and 
$$\delta_{i+1}=\delta_i-\min(\delta_i, \alpha-\delta_i).$$
We will show that $\delta_i$ is a decreasing sequence eventually stabilizing at
$$\delta=\begin{cases}
\alpha &\hbox{if } \beta=\alpha\\
0 & \hbox{otherwise}.
\end{cases}$$
It then follows that the exponent of $p$ in the prime decomposition of $\frac Ms$ is
$$\alpha-\delta=\begin{cases}
0 &\hbox{if } \beta=\alpha\\
\alpha & \hbox{otherwise}
\end{cases}$$ as needed.

For the claim about the sequence $\delta_i$, observe that $\delta_{i+1} \leq \delta_i$ by definition, and $\delta_i=\delta_{i+1}$ exactly if either $\delta_i=0$, or if $\delta_{i}=\alpha$. In the latter case, we have
$$\beta=\delta_0 \geq \delta_i =\alpha  \geq \beta$$
so we must have equality everywhere, thus $\delta_i=\alpha$ exactly if $\beta=\alpha$ and in this case $\delta=\alpha$, and otherwise we have $\delta=0$, which is what we had to prove.
\end{proof}

Combining the above lemma with Remark \ref{Rem: G_c computable}, we obtain Algorithm \ref{alg:order of cstables}, which on input $A,B,$ a simple cycle $c$ with $B_c \neq 0$, and the order $N$ of $(G_B)_{\source(c)}$, outputs the order of the maximal $c$-stable subgroup $G_c$.

\begin{algorithm}
	\caption{Computing the order of the maximal $c$-stable subgroup}\label{alg:order of cstables}
	\begin{algorithmic}
		\Procedure{c-StableOrder}{$A,B,c, N$} 
			\State $r \gets 1$ \Comment{Calculating $M=|\Stab c|$}
			\For{$c'$ prefix of $c$} 
				\State $r \gets \lcm(r, \frac{A_{c'}}{\gcd(A_{c'}, |B_{c'}|)})$			
			\EndFor
		\State $M \gets \frac{N}{r}$	
		\State $k \gets \frac{M}{\gcd\left(M, \frac{B_c}{A_c}\right)}$ \Comment{Calculating $k$ generating $\ker(|_c)$}	
		
		\State $s \gets k$ \Comment{Calculating $n$}
			\While{$s \neq \frac{s}{\gcd(s, M/s)}$}
			$s \gets \frac{s}{\gcd(s, M/s)}$
			\EndWhile
		\State $n \gets \frac{M}{s}$	
		\State $K \gets \frac{M}{n}$ \Comment{The order of $|G_c|$}
		\State \textbf{return} $K$
		\EndProcedure
	\end{algorithmic}
\end{algorithm}

\begin{Proposition}
	\label{Prop: Alg3 poly space}
	
	Algorithm \ref{alg:order of cstables} runs in space polynomial in the size of $(A,B)$, in particular, the output $K$ has size polynomial in the the size of $(A,B)$.
\end{Proposition}

\begin{proof}
At any given time, we need to store (in addition to $(A,B)$):
\begin{itemize}
	\item  the cycle $c$ and its prefix $c'$, with size $O(\mathfrak N \log \mathfrak N)$;
	\item the variable $N$, which by Proposition \ref{Prop: Alg2 poly space} has polynomial size;
	\item the variables $A_c,B_c, A_{c'}, B_{c'}$, with size $O(\mathfrak N\log \mathfrak m)$;
	\item the variables $r, M, k, s,n, K$ are all bounded by $\prod_{e \in \overline{E}_A} A_e \leq \mathfrak m ^ {{\mathfrak N}^2}$ so their size is $O({\mathfrak N}^2\log \mathfrak m)$.
\end{itemize}
The claim follows.
\end{proof}

We put these two procedures together to test $\SingularNotTightIdeal=\emptyset$ by Theorem \ref{Thm: testing singular elements faithful}. We can decide if a subgroup of order $k$ contains a free orbit by Corollary \ref{Cor: H supports singular characterization}.

The following procedure takes a Katsura pair $(A,B)$ as input and outputs YES if $\SingularNotTightIdeal=\emptyset$ in $KS_{( G_B, E_A)}$ and NO otherwise.

\begin{algorithm}
	\caption{Deciding if $\SingularNotTightIdeal=\emptyset$ for $( G_B, E_A)$}\label{alg: decide simplicity}
	\begin{algorithmic}
		
		\ForAll{$v \in E_A^0$} 
			\If{\textsc{Order}$(A,B,v)\neq 0$} \Comment{Checking if $(G_B)_v$ is finite}
			\State $N \gets $\textsc{Order}$(A,B,v)$ \Comment{Computing $N=|(G_B)_v|$}
				\ForAll{$c \in v\overline E_A^*v$ simple cycles with $B_c \neq 0$}
					\State $K \gets $\textsc{c-StableOrder}$(A,B,c,N)$ \LongComment{Computing the order of the maximal $c$-stable subgroup}
					\ForAll{$q \in v \overline E_A^*$ simple paths} 
						\If{$\frac{A_q}{\gcd(A_q, |B,q|)} = K$}\Comment{Checking if the orbit of $q$ is free}
						\State \textbf{continue} with next $c$
						\EndIf
					\State  halt and \textbf{return} NO \LongComment{We found a $c$-stable subgroup with no free orbits}
					\EndFor
					
				\EndFor		
			\EndIf
		\EndFor
		
		\State \textbf{return} YES

	\end{algorithmic}
\end{algorithm}

\begin{Proposition}
	\label{Prop: Alg4 poly space}
	
	Algorithm \ref{alg: decide simplicity} runs in space polynomial in the size of $(A,B)$.
\end{Proposition}

\begin{proof}
	At any given time, we need to store (in addition to $(A,B)$):
	\begin{itemize}
		\item the current vertex $v$ which has size $O(\log \mathfrak N)$;
		\item the order $N$ of $(G_B)_v$, which by Proposition \ref{Prop: Alg2 poly space} is polynomial size;
		\item the current cycle $c$ with size $O(\mathfrak N \log \mathfrak N)$;
		\item the variable $B_c$ with size $O(\mathfrak N\log \mathfrak m)$;
		\item the variable $K$, which by Proposition \ref{Prop: Alg3 poly space} is of polynomial size;
		\item the current path $q$ with size $O(\mathfrak N \log \mathfrak N)$;
		\item the variables $A_q$ and $B_q$ with size $O(\mathfrak N\log \mathfrak m)$.
	\end{itemize}
	The claim follows.
\end{proof}
\begin{Remark}\label{Rem: faithful simplicity algorithm}
	Recall that the simplicity of $K\G_{(G_B, E_A)}$ and $C^{\ast}_r(\G_{(G_B, E_A)})$ coincide by Corollary \ref{Cor: cstar simplicity iff steinberg simplicity}. Algorithm \ref{alg: decide simplicity} decides whether the singular ideal vanishes in polynomial space. This can be extended to a polynomial space algorithm to decide simplicity of $K\G_{(G_B, E_A)}$ by checking whether $\G_{(G_B, E_A)}$ is minimal and topologically free. Minimality is equivalent to $E_A$ satisfying condition (Cof), and because the action is faithful, topological freeness is equivalent to condition (Cyc). These can be checked in polynomial space using the algorithms described in Remark \ref{Rem: nonfaithful simplicity algorithm}.
	
\end{Remark}

%
\section{Examples}
We restrict our examples to those defined on strongly connected graphs with at least two simple cycles so that conditions (Cof) and (Cyc) (see Section \ref{Subsec: semigroup and groupoids algebras}) are always satisfied. It follows that for these examples $\G_{(\hat G_B, E_A)}$ and $\G_{(G_B, E_A)}$ are always minimal, and because $(G_B, E_A)$ acts faithfully, $\G_{(G_B, E_A)}$ is always topologically free. Establishing when $\G_{(\hat G_B, E_A)}$ is topologically free depends only on determining whether $(\hat G_B, E_A)$ satisfies (Evr), which can easily be done using Lemma \ref{lem: fixed path computation}.


First we recall notation. In each example, we denote the singular ideal and tight ideal of $KS_{(\hat G_B, E_A)}$ by $\hat I_K$ and $\hat T_K$ respectively, while the singular ideals of $K\G_{(\hat G_B, E_A)}$ and $C^{\ast}_r(\G_{(\hat G_B, E_A)})$ are denoted by $\hat J_K$ and $\hat J$ respectively. We use $I_K, T_K, J_K$ and $J$ for the respective ideals associated to $(G_B, E_A)$. We begin with an example for which the singular ideals vanish.

\begin{Example}\label{Ex: cases (1) and (3) are uninteresting}
	Consider the matrices
	\[A=\begin{pmatrix}
		2&2&0\\
		1&0&2\\
		0&1&0
	\end{pmatrix}, 
	B=\begin{pmatrix}
		3&1&0\\
		0&0&1\\
		0&0&0
	\end{pmatrix},\]
	and refer to Figure \ref{fig: First example graph}.
	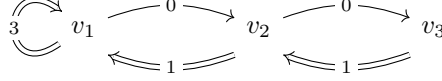
\begin{figure}[ht!]
		\centering
		\[\begin{tikzcd}
			{v_1} && {v_2} && {v_3}
			\arrow["3" description, Rightarrow, from=1-1, to=1-1, loop, in=145, out=215, distance=10mm]
			\arrow["0" description, bend left=20, from=1-1, to=1-3]
			\arrow["1" description, shift left=3, bend left=20, Rightarrow, from=1-3, to=1-1]
			\arrow["0" description, bend left=20, from=1-3, to=1-5]
			\arrow["1" description, shift left=3, bend left=20, Rightarrow, from=1-5, to=1-3]
		\end{tikzcd}\]
		\caption{The graph $E_A$ of Example \ref{Ex: cases (1) and (3) are uninteresting} edges labeled by $B$ values and edge multiplicity indicated by parallel stems.}
		\label{fig: First example graph}
	\end{figure}
	We begin by analyzing $(\hat G_B, E_A)$. Let $c$ be the loop at $v_1$. Observe that $\frac{B_{c}}{A_{c}}\notin \ZZ$, so by Proposition \ref{Prop: Infinite group characterization} the action of $(\hat G_B,)_{v_1}$ has a free orbit. It follows from Proposition \ref{Prop: free orbits prevent singular elements} that there are no nonzero singular elements in $K(\hat G_B)_{v_1}$. The isotropy groups $(\hat G_B)_{v_2}$ and $(\hat G_B)_{v_3}$ have non-faithful actions because each path $p$ from $v_1$ to either $v_2$ or $v_3$ has $B_p=0$, and $c$ is the only simple cycle with $\frac{B_c}{A_c}\notin \mathbb{Z}$. The isotropy groups $(\hat G_B)_{v_2}$ and $(\hat G_B)_{v_3}$ fail condition (T1) of Theorem \ref{Thm:unfaithful J=0}, so $\hat J_K=0=\hat J$. The graph is strongly connected and has more than one cycle so conditions (Cof) and (Cyc) are satisfied. It follows from Lemma \ref{lem: fixed path computation} that the existence of zero edges ensures that (Evr) is satisfied (see Remark \ref{Rem: nonfaithful simplicity algorithm}). We conclude that $\G_{(\hat G_B, E_A)}$ is minimal and topologically free so by Corollary \ref{Cor: cstar simplicity iff steinberg simplicity} and Proposition \ref{Prop: simple implies purely infinite}, $K\G_{(\hat G_B, E_A)}$ and $C^{\ast}_r(\G_{(\hat G_B, E_A)})$ are purely infinite simple.
	
	We turn our attention to the faithful KEP-groupoid $(G_B, E_A)$. Since $K\G_{(\hat G_B, E_A)}$ and $C^{\ast}_r(\G_{(\hat G_B, E_A)})$ are simple, they naturally have to coincide with $K\G_{(G_B, E_A)}$ and $C^{\ast}_r(\G_{(G_B, E_A)})$. Nonetheless, we explain how Section \ref{Sec: faithful} yields their simplicity. As we have seen, $K(G_B)_{v_1}$ supports no nonzero singular elements. Observe that for $i=2,3$, if $c\in v_iE_A^*$ is a cycle, then $B_{c}=0$. In particular, $c$-stable subgroups of $(G_B)_{v_2}$ and $(G_B)_{v_3}$ are trivial. Therefore $I_K=T_K$ by Theorem \ref{Thm: testing singular elements faithful} and $J_K=0=J$. Again simplicity follows from the minimality and topological freeness of $\G_{(G_B, E_A)}$.
\end{Example}

The previous example shows that $\hat J_K$ and $J_K$ may vanish simultaneously. It turns out that vanishing of $\hat J_K$ and $J_K$ are independent of each other, as the next example shows (see Remark \ref{Rem: example 2 remark}). This should not be too surprising given that the vanishing of $\hat J_K$ depends primarily on zero edges whereas vanishing of $J_K$ is determined by non-zero edges.
\begin{Example}\label{Ex: example 2}
	Consider the matrices
	\[A=\begin{pmatrix}
		1&1&1&1\\
		3&0&0&0\\
		2&0&0&0\\
		1&0&0&2
	\end{pmatrix}, 
	B=\begin{pmatrix}
	b&3&2&0\\
	2&0&0&0\\
	3&0&0&0\\
	0&0&0&3
	\end{pmatrix},\]
	 with $b\in \{0,1,2,3,4,5\}$ and refer to Figure \ref{fig: Examples Graphs}. Once again, we begin with $(\hat G_B, E_A)$. Observe that $\frac{B_c}{A_c}\in\ZZ$ for each cycle $c$ excluding the loop at $v_4$, so by Proposition \ref{Prop: Infinite group characterization} only $(\hat G_B)_{v_4}$ acts faithfully. Furthermore there are infinite paths from $v_1, v_2$ and $v_3$ avoiding zero edges so $v_1, v_2$ and $v_3$ satisfy condition (T1) of Theorem \ref{Thm:unfaithful J=0}. The graph is strongly connected and contains a zero edge so $v_1, v_2$, and $v_3$ also satisfy condition (T2) of Theorem \ref{Thm:unfaithful J=0}, and therefore $\hat J_K\neq 0\neq \hat J$. This obstructs the simplicity of $K\G_{(\hat G_B, E_A)}$ and $C_r^*(\G_{(\hat G_B, E_A)})$. We note that because $E_A$ is strongly connected, $\G_{(\hat G_B, E_A)}$ is minimal. Furthermore, condition (Cyc) is satisfied, and by Lemma \ref{lem: fixed path computation} the presence of zero edges guarantees that condition (Evr) is satisfied, so $\G_{(\hat G_B, E_A)}$ is topologically free.
	 
	 Moving now to the faithful KEP-groupoid, applying Proposition \ref{Prop: Infinite group characterization}, we find that the finite isotropy groups of $G_B$ have orders
	\[|(G_B)_{v_1}|=\lcm \{2, 3\}=6,\]
	\[|(G_B)_{v_2}|=\lcm \{1, 2\}=2,\]
	\[|(G_B)_{v_3}|=\lcm \{1, 3\}=3.\]
	Because of their prime orders, $(G_B)_{v_2}$ and $(G_B)_{v_3}$ have free orbits, and therefore do not support singular elements which are not tight by Proposition \ref{Prop: free orbits prevent singular elements}. 
	
	All that remains is $(G_B)_{v_1}$. Let $c$ be the loop at $v_1$, and let $g$ generate $(G_B)_{v_1}\cong C_6$. Observe that the map $|_c: (G_B)_{v_1}\rightarrow (G_B)_{v_1}$ is simply given by $g^k\mapsto g^{bk}$ by Lemma \ref{lem: fixed path computation}. Therefore, the maximal $c$-stable subgroup $G_c$ is trivial when $b=0$, has order 2 when $b=3$, and order $3$ when $b\in \{2,4\}$. In each of these cases, $G_c$ has a free orbit on $v_1E^1$ and thus does not support a non-zero singular element. When $b\in \{1,5\}$, we find that $G_c=(G_B)_{v_1}$ which has no free orbit on $v_1E^*$ by Proposition \ref{Prop: Subgroups with free orbits characterization}.
	
	Applying Theorem \ref{Thm: testing singular elements faithful}, we conclude that for $b\in \{0,2,3,4\}$, we have $I_k=T_K$ and therefore $J_K=0=J$. When $b\in \{1,5\}$, we find that $(1-g)(1-g^2)(1-g^3)\in K(G_B)_{v_1}$ is singular but not tight by the same Theorem, so $J_K\neq 0\neq J$.
	
	The graph is strongly connected with at least two cycles so $\G_{(G_B, E_A)}$ is minimal and topologically free. Therefore when $b\in \{0,2,3,4\}$, $K\G_{(G_B, E_A)}$ and $C_r^*(\G_{(G_B, E_A)})$ are purely infinite simple by Theorem \ref{Thm: testing singular elements faithful}, Corollary \ref{Cor: cstar simplicity iff steinberg simplicity} and Proposition \ref{Prop: simple implies purely infinite}. If instead $b\in \{1,5\}$, we find that the singular ideals do not vanish and so neither $K\G_{(G_B, E_A)}$ nor $C_r^*(\G_{(G_B, E_A)})$ are simple.
\end{Example}
\begin{figure}[ht!]
	\centering
	\[\begin{tikzcd}
		&&& {v_1} &&& \\
		{v_2} &&&&&& {v_3} \\
		&&& {v_4}
		\arrow["b" description, shift left, from=1-4, to=1-4, loop, in=55, out=125, distance=10mm]
		\arrow["2" description, shift right, bend right=20, Rightarrow, scaling nfold=3, from=1-4, to=2-1]
		\arrow["3" description, shift left, bend left=20, Rightarrow, from=1-4, to=2-7]
		\arrow["0" description, shift right, bend right=20, from=1-4, to=3-4]
		\arrow["3" description, shift right, bend right=20, from=2-1, to=1-4]
		\arrow["2" description, shift left, bend left=20, from=2-7, to=1-4]
		\arrow["0" description, bend right=20, from=3-4, to=1-4]
		\arrow["3" description, shift right, Rightarrow, from=3-4, to=3-4, loop, in=305, out=235, distance=10mm]
	\end{tikzcd}\]
	\caption{The graph $E_A$ of Example \ref{Ex: example 2} with edges labeled by $B$ and edge multiplicity indicated by parallel stems.}
	\label{fig: Examples Graphs}
\end{figure}
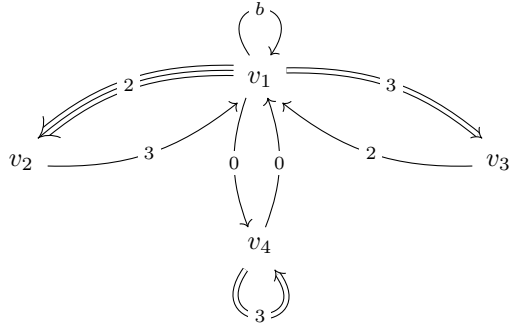
\begin{Remark}\label{Rem: example 2 remark}
	In the discussion before Example \ref{Ex: example 2}, we claim that the vanishing of $\hat J_K$ and $J_K$ are independent of each other. Our examples have shown all but one combination: Example \ref{Ex: cases (1) and (3) are uninteresting} displays $\hat J_K=0=J_K$ and Example \ref{Ex: example 2} displays $\hat J_K\neq0=J_K$ when $b\in \{1,2,3,4\}$ and $\hat J_K\neq0\neq J_K$ when $b\in\{1,5\}$. 
	
	An example exhibiting $\hat J_K=0\neq J_K$ is obtained by omitting $v_4$ (and adjacent edges) from Example \ref{Ex: example 2} and choosing $b\in \{1,5\}$. Let $(\tilde A, \tilde B)$ denote the resulting Katsura pair. The analysis of $(G_{\tilde B}, E_{\tilde A})$ exactly follows that of $(G_B, E_A)$ in Example \ref{Ex: example 2} and so $J_K\neq 0\neq J$, however, there are no zero edges in the graph, so by Theorem \ref{Thm:unfaithful J=0}, $\hat J_K=0 = \hat J$. In fact, $\G_{(\hat G_{\tilde B}, E_{\tilde A})}$ is Hausdorff by \cite[Thm. 18.6]{ExParSSGraphs}.
	
	Despite the singular ideals vanishing, $K\G_{(\hat G_{\tilde B}, E_{\tilde A})}$ and $C_r^\ast(\G_{(\hat G_{\tilde B}, E_{\tilde A})})$ must not be simple because they respectively quotient onto $K\G_{(G_{\tilde B}, E_{\tilde A})}$ and $C_r^\ast(\G_{( G_{\tilde B}, E_{\tilde A})})$ which are not simple because $J_K\neq 0\neq J$. Indeed, the non-faithful isotropies and the lack of zero edges imply that $(\hat G_{\tilde B}, E_{\tilde A})$ no longer satisfies (Evr) by Lemma \ref{lem: fixed path computation}, and therefore simplicity is obstructed because $\G_{(\hat G_{\tilde B}, E_{\tilde A})}$ fails to be topologically free.
\end{Remark}

\printbibliography
\end{document}